\documentclass[12pt]{amsart}
\usepackage{amssymb}
\usepackage{amsfonts}
\usepackage{latexsym}
\usepackage{amscd}
\usepackage[mathscr]{euscript}
\usepackage{xy} \xyoption{all}
\usepackage[active]{srcltx}
\usepackage{color,graphics}

\newcommand{\exend}{\unskip\nobreak\hfill$\blacktriangleleft$}

\usepackage{tikz,tkz-graph,tikz-cd}
\usetikzlibrary{matrix,arrows,calc,automata,trees}

\usepackage[shortlabels]{enumitem}
\usepackage{url}

\newcommand*{\LargerDot}{\scalebox{2}{$.$}}

\newcommand{\act}{\curvearrowright}
\newcommand{\andspace}{\quad \text{and} \quad}
\newcommand{\If}{\text{if }}
\newcommand{\ini}{\mathfrak{i}}
\newcommand{\ter}{\mathfrak{t}}
\newcommand{\niceblue}{rgb:red,1;green,2;blue,3}
\newcommand{\nicegreen}{rgb:red,1;green,4;blue,2}
\newcommand{\nicered}{rgb:red,4;green,1;blue,2}

\newcommand{\bigslant}[2]{{\raisebox{.1em}{$#1 \hspace{-2pt}$}\left/\raisebox{-.1em}{$#2$}\right.}}
\DeclareMathOperator{\id}{id} 

\theoremstyle{plain}
\newtheorem{theorem}{Theorem}[section]
\newtheorem{lemma}[theorem]{Lemma}
\newtheorem{proposition}[theorem]{Proposition}
\newtheorem{corollary}[theorem]{Corollary}

\theoremstyle{definition}
\newtheorem{definition}[theorem]{Definition}
\newtheorem{example}[theorem]{Example}

\newtheorem*{cexample}{Example~\ref{ex} (continued)} 

\newtheorem{remark}[theorem]{Remark}
\newtheorem*{remark*}{Remark}

\newtheorem*{assumption*}{Assumption}

\begin{document}
\allowdisplaybreaks
\null\vskip-1.5cm

\title[Nuclearity and exactness]{On nuclearity and exactness of the tame $C^*$-algebras associated with finitely separated graphs}%
\author{Matias Lolk}
\address{Department of Mathematical Sciences, University of Copenhagen, 2100 Copenhagen, Denmark.}\email{lolk@math.ku.dk}
\thanks{Supported by the Danish National Research Foundation through the Centre for Symmetry and Deformation (DNRF92)}
\date{\today}

\begin{abstract}
We introduce a graph theoretic property called Condition (N) for finitely separated graphs and prove that it is equivalent to both nuclearity and exactness of the associated universal tame graph $C^*$-algebra.
\end{abstract}

\maketitle

\section*{Introduction}

A \textit{finitely separated graph} is a directed graph with a partition of the edges into finite subsets, which might be thought of as an edge colouring, so that edges with distinct ranges have different colours. To any such graph $(E,C)$, Ara and Exel introduced a $C^*$-algebra $\mathcal{O}(E,C)$ in \cite{AE,AE2}, referred to as the universal \textit{tame} $C^*$-algebra of $(E,C)$. It is generated by the vertices and the edges of the graph with relations similar to the ordinary Cuntz-Krieger relations, taking into account the colouring, so that the edge set $E^1$ defines a tame set of partial isometries. Among many other results, they provided a very useful dynamical description of $\mathcal{O}(E,C)$ when $(E,C)$ is finite and bipartite, and this was generalised to finitely separated graphs by the author in \cite{Lolk}. Specifically, $\mathcal{O}(E,C)$ may be identified with a universal crossed product $C_0(\Omega(E,C)) \rtimes \mathbb{F}$ for a partial action $\theta^{(E,C)}$ of a free group $\mathbb{F}=\mathbb{F}(E^1)$ on a locally compact, zero-dimensional Hausdorff space $\Omega(E,C)$. The potency of these partial actions was immediately demonstrated when they were used to answer a question of R\o rdam and Sierakowski \cite{RS} about relative type semigroups in the negative \cite[Section 7]{AE}. 
\medskip \\
Recently, Ara and the author have introduced the more general notion of a \textit{convex subshift} \cite[Section 3]{AL} and shown that all convex subshifts of \textit{finite type} arise, up to Kakutani equivalence, as the partial action associated with a finite bipartite graph. As such, finitely separated graphs may be viewed as combinatorial models for a wide class of partial actions, which includes both one-sided and two-sided shift spaces, but which also includes many new and previously unstudied dynamical systems. \medskip \\
Nuclearity and exactness of $\mathcal{O}(E,C)$ have not been systematically studied before, but both nuclearity and non-exactness have been observed in a number of examples. If $E$ is any column-finite graph, it may be regarded as a finitely separated graph by equipping it with the \textit{trivial separation} $\mathcal{T}$, and in this case, the tame $C^*$-algebra $\mathcal{O}(E,\mathcal{T})$ is simply the classical graph $C^*$-algebra $C^*(E)$, hence nuclear. Likewise, if $\mathcal{X}$ is any two-sided subshift of finite type, then the crossed product $C(\mathcal{X}) \rtimes \mathbb{Z}$ is Morita equivalent to $\mathcal{O}(E,C)$ for an appropriate finite bipartite separated graph by \cite[Proposition 6.8]{AL}, so in this case $\mathcal{O}(E,C)$ is nuclear as well. In the other direction, one can easily identify $C^*(\mathbb{F}_n)$ with a tame separated graph $C^*$-algebra: the graph is nothing but a single vertex with $n$ loops of different colours. A much more interesting class of non-exact examples were considered by Ara, Exel and Katsura in  \cite{AEK}, where they studied $C^*$-algebras $\mathcal{O}_{m,n}$ for $2 \le m < n < \infty$ that might be considered as two-parameter versions of the Cuntz algebras. \medskip \\
In this paper, we characterise nuclearity and exactness in terms a graph-theoretic property that we call Condition (N). To this end, we first introduce a notion of topological amenability for partial actions, and prove, using the theory of groupoid $C^*$-algebras, that a partial action of a discrete group on a locally compact Hausdorff space is topologically amenable if and only if the corresponding crossed products are nuclear (Theorem~\ref{thm:TopAmenThm}). Given any finitely separated graph $(E,C)$, we then identify two complementary subgraphs $(E_{\textup{Br}}, C^{\textup{Br}})$ and $(E_{\textup{BF}}, C^{\textup{BF}})$, called the \textit{branching subgraph} and the \textit{branch free subgraph}, respectively. The $C^*$-algebra $\mathcal{O}(E_{\textup{Br}},C^{\textup{Br}})$ naturally appears as a quotient of $\mathcal{O}(E,C)$, and in Section 3, we prove that both nuclearity and exactness of $\mathcal{O}(E_{\textup{Br}},C^{\textup{Br}})$ is equivalent to Condition (N). We can use the approach of \cite{AEK} to establish necessity (Proposition~\ref{prop:Exactness}), but sufficiency takes more work: The main step is the construction of a \textit{proper orientation} of the branching subgraph (Theorem~\ref{thm:SpGrFølner}) in the presence of Condition (N). In Section 4, we prove that $\mathcal{O}(E_{\textup{BF}},C^{\textup{BF}})$ is \textit{always} nuclear, and that nuclearity of $\mathcal{O}(E_{\textup{Br}},C^{\textup{Br}})$ and $\mathcal{O}(E_{\textup{BF}},C^{\textup{BF}})$ implies nuclearity of $\mathcal{O}(E,C)$, before putting everything together in section 5 to obtain our main theorem (Theorem~\ref{thm:MainThm}). We finally study a number of examples.

\section{Preliminary definitions}
In this section, we recall the necessary definitions and results from the existing theory on algebras associated with separated graphs in a slightly condensed version -- the reader may consult \cite[Section 2]{AL} and \cite[Section 1]{Lolk} for more details. Most importantly, we describe the $C^*$-algebra $\mathcal{O}(E,C)$ as a universal crossed product for a partial action.

\begin{definition}
A \textit{finitely separated graph} $(E,C)$ is a graph $E=(E^0,E^1,r,s)$ together with a \textit{separation} $C=\bigsqcup_{v \in E^0} C_v$, where each $C_v$ is a partition of $r^{-1}(v)$ into non-empty finite subsets. In case $r^{-1}(v)=\emptyset$, we simply take $C_v$ to be the empty partition, and for any edge $e \in E^1$, we will denote the element of $C$ containing $e$ by $[e]$.
\end{definition}

\begin{example}\label{ex}
Below is an example of a finite separated graph:
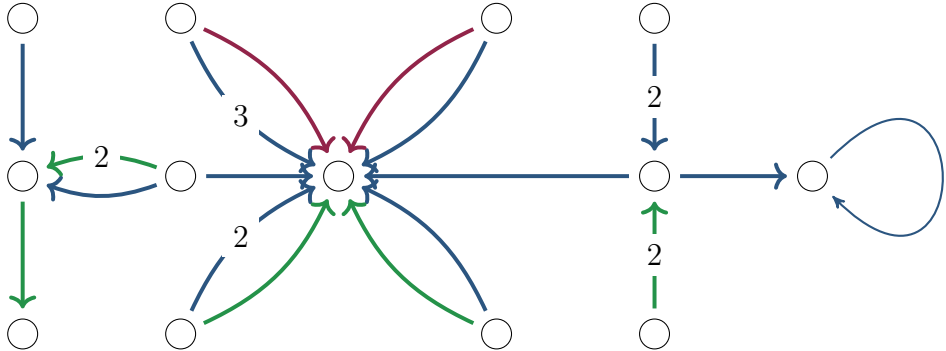
\begin{center}
\begin{tikzpicture}[scale=0.70]
 \SetUpEdge[lw         = 1.5pt,
            labelcolor = white]
  \tikzset{VertexStyle/.style = {draw,shape = circle,fill = white, inner sep=4pt,minimum size=10pt,outer sep=4pt}}

  \SetVertexNoLabel

  \Vertex[x=-6,y=0]{u1}  
  \Vertex[x=-3,y=0]{u2}
  \Vertex[x=0,y=0]{u3}
  \Vertex[x=6,y=0]{u4}
  \Vertex[x=9,y=0]{u5}
  
  \Vertex[x=-3,y=3]{u7}
  \Vertex[x=3,y=3]{u8}
  \Vertex[x=-6,y=3]{u6}
  \Vertex[x=6,y=3]{u9}

  \Vertex[x=-6,y=-3]{u10}  
  \Vertex[x=-3,y=-3]{u11}
  \Vertex[x=3,y=-3]{u12}
  \Vertex[x=6,y=-3]{u13}

  \tikzset{EdgeStyle/.style = {->,color={\niceblue}}}
  \Loop[dir=EA](u5)  
  \Edge[](u2)(u3)
  \Edge[](u4)(u3)
  \Edge[](u4)(u5)
  \Edge[](u6)(u1)
  \Edge[label=$2$](u9)(u4)
  
  \tikzset{EdgeStyle/.style = {->,bend left=20,color={\niceblue}}}
  \Edge[](u2)(u1)
  \Edge[](u8)(u3)
  \Edge[label=$2$](u11)(u3)
  
  \tikzset{EdgeStyle/.style = {->,bend right=20,color={\niceblue}}}
  \Edge[label=$3$](u7)(u3)
  \Edge[](u12)(u3)
  
  \tikzset{EdgeStyle/.style = {->,color={\nicegreen}}}
  \Edge[label=$2$](u13)(u4)
  \Edge[](u1)(u10)
  
  \tikzset{EdgeStyle/.style = {->,bend left=20,color={\nicegreen}}}
  \Edge[](u12)(u3)
  
  \tikzset{EdgeStyle/.style = {->,bend right=20,color={\nicegreen}}}
  \Edge[](u11)(u3)
  \Edge[label=$2$](u2)(u1)
  
  \tikzset{EdgeStyle/.style = {->,bend left=20,color={\nicered}}}
  \Edge[](u7)(u3)
  
  \tikzset{EdgeStyle/.style = {->,bend right=20,color={\nicered}}}
  \Edge[](u8)(u3)
\end{tikzpicture}
\end{center}
We gladly use the same colour for edges with different ranges when depicting separated graphs -- otherwise one would need nine colours here -- so the colouring should only be understood as a partition of the edges going into a given vertex. The numbering indicates the number of edges, so that we may simply write a number, say $42$, instead of visually representing each of the $42$ edges. This particular graph will serve as our main example throughout the paper. \exend
\end{example}

Recall that a set of partial isometries $S$ is called \textit{tame} if every product formed from $S \cup S^*$ is again a partial isometry.

\begin{definition}[{\cite[Definition 2.4]{AE2}}]
Let $(E,C)$ denote a finitely separated graph. The \textit{universal tame graph} $C^*$\textit{-algebra} $\mathcal{O}(E,C)$ is the universal $C^*$-algebra generated by $E^0 \sqcup E^1$ with relations
\begin{enumerate}[leftmargin=2cm,rightmargin=2cm]
\item[(V)] $uv = \delta_{u,v}v$ and $u=u^*$ for $u,v \in E^0$,
\item[(E)] $es(e)=r(e)e=e$ for $e \in E^1$,
\item[(SCK1)] $e^*f = \delta_{e,f}s(e)$ if $[e]=[f]$,
\item[(SCK2)] $v = \sum_{e \in X}ee^*$ for all $v \in E^0$ and $X \in C_v$,
\item[(T)] $E^1 \subset \mathcal{O}(E,C)$ is tame.
\end{enumerate}
The reader should note that we use the convention of \cite{AE}, \cite{AE2}, \cite{AL} and \cite{Lolk}, often referred to as the \textit{Raeburn-convention}, opposite to the one used in \cite{AG} and \cite{AG2}.
\end{definition}

\begin{remark}
A subgraph $(F,D)$ of $(E,C)$ is called \textit{complete} if 
$$D_v=\{X \in C_v \mid X \cap F^1 \ne \emptyset \}$$
for every $v \in F^0$, and by universality there is an induced $*$-homomorphism $\mathcal{O}(F,D) \to \mathcal{O}(E,C)$. Since every finitely separated graph is the direct limit of its finite complete subgraphs (see \cite[Section 3]{AG2} for the precise meaning of this), and $\mathcal{O}$ is a continuous functor from the category of finitely separated graphs \cite[Proposition 7.2]{AE2}, we see that $\mathcal{O}(E,C)$ may always be approximated by $C^*$-algebras $\mathcal{O}(F,D)$ for $(F,D)$ a finite separated graph. \exend
\end{remark}

We now recall a bit of terminology related to separated graphs from \cite{Lolk}.

\begin{definition}
Let $(E,C)$ denote a (finitely) separated graph. The \textit{double} $\widehat{E}$ of $E$ is the graph given by $\widehat{E}^0:=E^0$ and $\widehat{E}^1=E^1 \sqcup \{e^{-1} \mid e \in E^1\}$ with range and source maps extended by $r(e^{-1}):=s(e)$ and $s(e^{-1})=r(e)$. An \textit{admissible path} $\alpha$ in $(E,C)$ is a path (read from the right) in $\widehat{E}$ such that
\begin{enumerate}
\item any subpath $ef^{-1}$ with $e,f \in E^1$ satisfies $e \ne f$,
\item any subpath $e^{-1}f$ with $e,f \in E^1$ satisfies $[e] \ne [f]$.
\end{enumerate}
We regard the vertices $v \in E^0$ as the \textit{trivial} admissible paths with $r(v):=v=:s(v)$, and if $\alpha$ is a non-trivial admissible path, we will write $\ini_d(\alpha)$ and $\ter_d(\alpha)$ for the initial and terminal symbol of $\alpha$, respectively; for instance 
$$\ini_d(ef^{-1})=f^{-1} \andspace \ter_d(ef^{-1})=e.$$
We then extend the range and source functions to admissible paths by the formulas $r(\alpha):=r(\ter_d(\alpha))$ and $s(\alpha):=s(\ini_d(\alpha))$. For admissible paths $\alpha$ and $\beta$, we will denote the concatenation by $\beta\alpha$, and we will always write $\beta \cdot \alpha$ if we allow for cancellation of mutual inverses. 
\medskip \\
A \textit{closed path} in $(E,C)$ is a non-trivial admissible path $\alpha$ with $r(\alpha)=s(\alpha)$, and $\alpha$ is called a \textit{cycle} if the concatenation $\alpha\alpha$ is an admissible path as well. Either way, we shall say that $\alpha$ is \textit{based} at $r(\alpha)=s(\alpha)$, and $\alpha$ is called \textit{base-simple} if $r(\beta) \ne s(\alpha)$ for all proper subpaths $s(\alpha) < \beta < \alpha$. We finally define a natural partial order $\le$ on the set of admissible paths by
$$\beta \le \alpha \Leftrightarrow \beta \text{ is an initial subpath of } \alpha,$$
and whenever $s(\alpha)=s(\beta)$, we write $\alpha \wedge \beta \le \alpha, \beta$ for the maximal initial subpath. \exend
\end{definition}

The notion of a \textit{partial action} is fundamental to this work, and so we briefly recall the essentials.

\begin{definition}
A partial action $\theta \colon G \act \Omega$ of a discrete group $G$ on a topological space $\Omega$ is a family of homeomorphisms of open subspaces $\{\theta_g \colon \Omega_{g^{-1}} \to \Omega_g\}_{g \in G}$, such that
\begin{itemize}
\item $\theta_g(\Omega_{g^{-1}} \cap \Omega_h) \subset \Omega_{gh}$ for all $g,h \in G$,
\item $\theta_g(\theta_h(x))=\theta_{gh}(x)$ for all $g,h \in G$ and $x \in \Omega_{h^{-1}} \cap \Omega_{h^{-1}g^{-1}}$.
\end{itemize}
For any point $x \in \Omega$, we will write $G^x:=\{g \in G \mid x \in \Omega_{g^{-1}}\}$ for the group elements that can act on $x$. Now if $\theta' \colon G \act \Omega'$ is another partial action, then a map $\varphi \colon \Omega \to \Omega'$ is called $G$\textit{-equivariant} if
\begin{itemize}
\item $\varphi(\Omega_g) \subset \Omega_g'$ for all $g \in G$,
\item $\theta_g'(\varphi(x))=\varphi(\theta_g(x))$ for all $g \in G$ and $x \in \Omega_{g^{-1}}$.
\end{itemize}
Similarly to the above, one can define the concept of a partial action on a $C^*$-algebra, demanding that the domains should be closed two-sided ideals. Hence, if $\Omega$ is a locally compact Hausdorff space, then $\theta$ translates into a partial $C^*$-action $\theta^* \colon G \act C_0(\Omega)$. As is the case for global actions, one can associate both a \textit{full} and a \textit{reduced} crossed product to a partial action, and there is a canonical surjective $*$-homomorphism
$$C_0(\Omega) \rtimes G \to C_0(\Omega) \rtimes_r G,$$
called the \textit{regular representation}. We refer the reader to \cite{Exel} for a comprehensive treatment of the theory of partial actions and their crossed products. \exend
\end{definition}

We now proceed to describe the partial dynamical system associated with a finitely separated graph, introduced in \cite{AE} for finite bipartite graphs and extended to the more general setting in \cite{Lolk}.

\begin{definition}[{\cite[Definition 2.6]{Lolk}}]\label{def:DynamicalPicture}
Suppose that $(E,C)$ is a finitely separated graph, and let $\mathbb{F}$ denote the free group on $E^1$. Given $\xi \subset \mathbb{F}$ and $\alpha \in \xi$, the \textit{local configuration} $\xi_\alpha$ of $\xi$ at $\alpha$ is the set
$$\xi_\alpha:=\{s \in E^1 \sqcup (E^1)^{-1} \mid s \in \xi \cdot \alpha^{-1}\}.$$
Then $\Omega(E,C)$ is the disjoint union of the discrete space $E_{\textup{iso}}^0$ and the set of $\xi \subset \mathbb{F}$ satisfying the following:
\begin{enumerate}
\item[(a)] $1 \in \xi$.
\item[(b)] $\xi$ is \textit{right-convex}: In view of (a), this exactly means that if $e_n^{\varepsilon_n} \cdots e_1^{\varepsilon_1} \in \xi$ for $e_i \in E^1$ and $\varepsilon_i \in \{\pm 1\}$, then $e_m^{\varepsilon_m} \cdots e_1^{\varepsilon_1} \in \xi$ as well for any $1 \le m < n$.
\item[(c)] For every $\alpha \in \xi$, there is some $v \in E^0$ and $e_X \in X$ for each $X \in C_v$, such that 
$$\xi_\alpha = s^{-1}(v) \sqcup \{e_X^{-1} \mid X \in C_v \}.$$
\end{enumerate}
$\Omega(E,C)$ is made into a topological space by regarding it as a subspace $\{0,1\}^\mathbb{F} \sqcup E_{\textup{iso}}^0$. One can easily check that it becomes a zero-dimensional, locally compact Hausdorff space, which is compact if and only if $E^0$ is a finite set. A topological partial action $\theta=\theta^{(E,C)} \colon \mathbb{F} \act \Omega(E,C)$ with compact-open domains is then defined by setting
$$\Omega(E,C)_\alpha := \{\xi \in \Omega(E,C) \mid \alpha^{-1} \in \xi\} \andspace \theta_\alpha(\xi) := \xi \cdot \alpha^{-1}$$
for $\alpha \in \mathbb{F}$ and $\xi \in \Omega(E,C)_{\alpha^{-1}}$. It follows from (a), (b) and (c) above that $\Omega(E,C)_\alpha$ is non-empty if and only if $\alpha$ is an admissible path. We finally set $\Omega(E,C)_{s(e)}:= \Omega(E,C)_{e^{-1}}$ for every $e \in E^1$ and
$$\Omega(E,C)_v := \bigsqcup_{e \in X}\Omega(E,C)_e$$
for every  $X \in C_v$. Note that in case $X \in C_v$ and $v=s(e)$, these two definitions coincide. If $v$ is isolated, we simply set $\Omega(E,C)_v:=\{v\}$. The reader may think of $\Omega(E,C)_v$ as the set of configurations ``starting'' in $v$, and we have $\Omega(E,C)=\bigsqcup_{v \in E^0} \Omega(E,C)_v$. \exend
\end{definition}

We will also need the notion of an \textit{animal}.

\begin{definition}[{\cite[Definition 2.9]{Lolk}}]\label{def:Animal}
An $(E,C)$\textit{-animal} is a right-convex subset $\omega \subset \xi$ of a configuration $\xi \in \Omega(E,C) \setminus E_{\textup{iso}}^0$ such that $\{1\} \subsetneq \omega$. It is called finite if it is a finite set, and we can define a compact subset of $\Omega(E,C)$ by
$$\Omega(E,C)_\omega := \{\xi \in \Omega(E,C) \mid \omega \subset \xi\},$$
which is open if $\omega$ is finite. It is easy to check that if $\{1\} \ne S \subset \mathbb{F}$ is any non-empty subset such that $\alpha \cdot \beta^{-1}$ is admissible for distinct $\alpha,\beta \in S \cup \{1\}$, then the right-convex closure $\langle S \rangle:=\text{conv}(S \cup \{1\})$ of $S \cup \{1\}$ inside $\mathbb{F}$ defines an $(E,C)$-animal. We warn the reader that we have the somewhat confusing identity $\Omega(E,C)_\alpha=\Omega(E,C)_{\{\alpha^{-1}\}}$. 
\end{definition}

Our main tool for studying the tame universal $C^*$-algebra of a separated graph is the following result, which was first obtained for finite bipartite separated graphs in \cite{AE}.

\begin{theorem}[{\cite[Theorem 2.10]{Lolk}}]
Let $(E,C)$ denote a finitely separated graph. Then there is a canonical isomorphism of $C^*$-algebras $\mathcal{O}(E,C) \cong C_0(\Omega(E,C)) \rtimes \mathbb{F}$. \exend
\end{theorem}

As a consequence of this theorem, there is also a natural \textit{reduced} tame $C^*$-algebra.

\begin{definition}[{\cite[Definition 6.8]{AE2}}]
\textit{The reduced tame graph $C^*$-algebra} of a finitely separated graph $(E,C)$ is the reduced crossed product $\mathcal{O}^r(E,C) := C_0(\Omega(E,C)) \rtimes_{r} \mathbb{F}$.
\exend
\end{definition}

Just as for non-separated graphs, there are certain sets of vertices that naturally correspond to ideals.

\begin{definition}[{\cite[Definition 6.3 and Definition 6.5]{AG2}}]
Let $(E,C)$ denote a finitely separated graph. Given any subset of vertices $H \subset E^0$, the \textit{full subgraph} $(E_H,C^H)$ of $H$ is given by $E_H^0:=H$, 
$$E_H^1:=\{e \in E^1 \mid r(e),s(e) \in H\}$$
with restricted range and source maps, and separation $C^H:=\{X \cap E_H^1 \mid X \cap E_H^1 \ne \emptyset\}$. The set $H$ is called \textit{hereditary} if $r(e) \in H$ implies $s(e) \in H$ for any $e \in E^1$, and it is called $C$\textit{-saturated} if $s(X) \subset H$ for $X \in C_v$ implies $v \in H$. The set of hereditary and $C$-saturated subsets $H \subset E^0$ is denoted $\mathcal{H}(E,C)$, and for all $H \in \mathcal{H}(E,C)$, we also have a \textit{quotient graph} $(E/H,C/H)$; this is simply the full subgraph of $E^0 \setminus H$. \exend
\end{definition}

Modding out the ideal generated by a hereditary and $C$-saturated subset, one obtains the tame graph $C^*$-algebra of the corresponding quotient graph.

\begin{theorem}[{\cite[Theorem 5.5]{AL}, \cite[Theorem 2.19]{Lolk}}]\label{thm:Quotient}
Let $(E,C)$ denote a finitely separated graph and consider any $H \in \mathcal{H}(E,C)$. The ideal $I(H)$ in $\mathcal{O}(E,C)$ generated by $H$ is the ideal induced from the open and invariant subspace $U(H):=\theta_\mathbb{F} \big(\bigsqcup_{v \in H} \Omega(E,C)_v \big)$, and there is a canonical isomorphism $\mathcal{O}(E,C)/I(H) \cong \mathcal{O}(E/H,C/H)$. \exend
\end{theorem}

Recall that in the non-separated setting, the ideal $I(H)$ corresponding to a hereditary and saturated set canonically contains $C^*(E_H)$ as a full corner. This fails in the separated setting for two reasons:

\begin{enumerate}
\item[(1)] The canonical map 
$$p \colon \bigsqcup_{v \in H} \Omega(E,C)_v \to \Omega(E_H,C^H)$$
of \cite[Remark 2.8]{Lolk} is usually not injective.
\item[(2)] There are usually admissible paths $\alpha$ with $r(\alpha),s(\alpha) \in H$, which are not entirely contained in $(E_H,C^H)$.
\end{enumerate}

In section~\ref{sect:BF}, we will consider certain hereditary and $C$-saturated subsets $H$, where situation (2) does not occur. We will then show that nuclearity of $\mathcal{O}(E,C)$ can be inferred from nuclearity of $\mathcal{O}(E_H,C^H)$ and $\mathcal{O}(E/H,C/H)$ in spite of (1).

\section{Topologically amenable partial actions}
In this section, we define the notion of \textit{topological amenability} for partial actions of discrete groups on locally compact Hausdorff spaces. Using the machinery of groupoids, we then check that our definition is equivalent to nuclearity of the crossed product $C^*$-algebras. The author would like to thank Claire Anantharaman-Delaroche for suggesting a groupoid approach to this problem.
\medskip \\
We first present the necessary definitions from groupoid theory. The reader is referred to \cite{Renault3} for a comprehensive treatment of topological groupoids and their $C^*$-algebras.
\begin{definition}
A groupoid $\mathcal{G}$ is a set with a distinguished subset $\mathcal{G}^{(0)} \subset \mathcal{G}$, \textit{range} and \textit{source} maps $r,s \colon \mathcal{G} \to \mathcal{G}^{(0)}$ and a partial composition
$$\{(\alpha,\beta) \in \mathcal{G} \times \mathcal{G} \mid s(\alpha)=r(\beta)\}=\mathcal{G}^{(2)}\ni (\alpha,\beta) \mapsto \alpha \cdot \beta,$$
such that
\begin{enumerate}[(1)]
\item $r(\alpha \cdot \beta)=r(\alpha)$ and $s(\alpha \cdot \beta)=s(\beta)$ for $(\alpha,\beta) \in \mathcal{G}^{(2)}$,
\item $r(\alpha)=\alpha=s(\alpha)$ for all $\alpha \in \mathcal{G}^{(0)}$,
\item $r(\alpha) \cdot \alpha = \alpha = \alpha \cdot s(\alpha)$ for all $\alpha \in \mathcal{G}$,
\item $(\alpha \cdot \beta) \cdot \gamma = \alpha \cdot (\beta \cdot \gamma)$ for $(\alpha,\beta),(\beta,\gamma) \in \mathcal{G}^{(2)}$,
\item for all $\alpha \in \mathcal{G}$, there is $\alpha^{-1} \in \mathcal{G}$ such that 
$$\alpha \cdot \alpha^{-1} = r(\alpha) \andspace \alpha^{-1} \cdot \alpha = s(\alpha).$$
\end{enumerate}
We shall use the fiber notation $\mathcal{G}^x:=r^{-1}(x)$ for $x \in \mathcal{G}^{(0)}$. \exend
\end{definition}
Alternatively, a groupoid can be defined as a small category in which all morphisms are invertible: Letting $\mathcal{G}$ denote the collection of morphisms and identifying the objects $\mathcal{G}^{(0)}$ of the category with the collection of identity morphisms, we have natural range and source maps as above, an associative partial composition and inverses. However, for our purposes it is easier to stress all the axioms explicitly as we shall immediately impose extra structure.

\begin{definition}
A \textit{topological groupoid} $\mathcal{G}$ is a groupoid equipped with a topology such that all the operations (i.e.~range, source, composition and inversion) are continuous when $\mathcal{G}^{(0)} \subset \mathcal{G}$ is given the subspace topology. Moreover, $\mathcal{G}$ is called an {\'e}tale groupoid, if $r$ and $s$ are local homeomorphisms. \exend
\end{definition}

The reason for passing to groupoids is that we can encode a partial action into a groupoid.

\begin{example}[Transformation groupoid]
Consider a partial topological action $\theta \colon G \act \Omega$ of a discrete group $G$ on a locally compact Hausdorff space $\Omega$. To any such partial action, we can associate a groupoid $\mathcal{G}_\theta$ as follows: Set
$$ \mathcal{G}_\theta :=\{(g,x) \in G \times \Omega \mid x \in \Omega_{g^{-1}}\} \quad , \quad
\mathcal{G}_\theta^{(0)} := \{1\} \times \Omega ,$$
and define range and source maps $r,s \colon \mathcal{G}_\theta \to \mathcal{G}_\theta^{(0)}$ by
$$r(g,x):=(1,\theta_g(x)) \andspace s(g,x):=(1,x).$$
Then $\alpha=(g,x) \in \mathcal{G}_\theta$ and $\beta=(h,y) \in \mathcal{G}_\theta$ are composable if and only if $x=\theta_h(y)$, in which case we set
$$\alpha \cdot \beta =(g,x) \cdot (h,y):=(gh,y).$$
We define an inversion by $(g,x)^{-1}=(g^{-1},\theta_g(x))$, and giving $\mathcal{G}_\theta$ the subspace topology of $G \times \Omega$, it is easily checked that it becomes an {\'e}tale locally compact Hausdorff groupoid, called the \textit{transformation groupoid} of $\theta$. In the future we shall always identify $\mathcal{G}^{(0)}=\{1\} \times \Omega$ and $\Omega$. \exend
\end{example}

In the following, whenever $\mu$ is a measure supported on $\mathcal{G}^{s(\alpha)}$, $\alpha \cdot \mu$ is the measure supported on $\mathcal{G}^{r(\alpha)}$ defined by $\alpha \cdot \mu(A):=\mu(\alpha^{-1} \cdot A)$.

\begin{definition}[\text{\cite[Definition 2.6]{Renault}}]\label{def:TopAmenDefGr}
Let $\mathcal{G}$ denote a locally compact Hausdorff groupoid. A \textit{topological approximate invariant mean} on $\mathcal{G}$ is a net $(m_i)_{i \in I}$, where each $m_i$ is a family $(m_i^x)_{x \in \mathcal{G}^{(0)}}$, $m_i^x$ being a probability measure on $\mathcal{G}^x=r^{-1}(x)$, such that
\begin{enumerate}
\item for all $i \in I$ and $f \in C_c(\mathcal{G})$, the map $\mathcal{G}^{(0)} \ni x \mapsto \int_{\mathcal{G}^x} f~dm_i^x$ is continuous,
\item $\sup_{\alpha \in \mathcal{K}} \Vert \alpha \cdot m_i^{s(\alpha)} - m_i^{r(\alpha)} \Vert \to 0$ for all compact subsets $\mathcal{K} \subset \mathcal{G}$,
\end{enumerate}
where the norm expression of (2) denotes the total variation of the measures. The groupoid $\mathcal{G}$ is said to be \textit{topologically amenable} if it admits a topological approximate invariant mean.
\end{definition}

\begin{remark}\label{rem:BODef}
Topological amenability has a number of different definitions in the literature, one being the existence of a \textit{topological invariant density} (see \cite[Definition 2.7]{Renault}). This is also the approach in \cite[Section 5.6]{BO}, but by \cite[Proposition 2.2.13]{DR} these two definitions are equivalent in the presence of a \textit{continuous Haar system}, which $\mathcal{G}_\theta$ always possesses \cite[Proposition 2.2]{Abadie2}. \exend
\end{remark}

We now specialise to partial actions of a discrete group $G$ on a locally compact Hausdorff space. To this end, define
$$\text{Prob}(G):=\{\mu \in \ell^1(G) \mid \mu \ge 0 \text{ and } \Vert \mu \Vert_1=1\}.$$
Then $G$ acts on $\text{Prob}(G)$ by $g.\mu(h):=\mu(hg)$.
\begin{definition}\label{def:TopAmenDef}
Consider a partial action $\theta \colon G \act \Omega$ of a discrete group $G$ on a locally compact Hausdorff space $\Omega$. A \textit{topological approximate invariant mean} for $\theta$ is a net $(m_i)_{i \in I}$, where each $m_i$ is a family $(m_i^x)_{x \in \Omega}$ with $m_i^x \in \text{Prob}(G)$ supported on $G^x=\{g \in G \mid x \in \Omega_{g^{-1}}\}$, such that
\begin{enumerate}
\item for all $i \in I$ and $g \in G$, the map $\Omega_{g^{-1}} \ni x \mapsto m_i^x(g)$ is continuous,
\item $\sup_{x \in K} \Vert g.m_i^x-m_i^{\theta_g(x)} \Vert_1 \to 0$ for any $g \in G$ and all compact subsets $K \subset \Omega_{g^{-1}}$,
\end{enumerate}
where the norm expression in (2) is the usual norm on $\ell^1(G)$. The partial action is said to be \textit{topologically amenable} if it admits a topological approximate invariant mean. \exend
\end{definition}

Note that for global actions, the above definition of topological amenability is equivalent to the classical one, see for instance \cite[Definition 4.3.5]{BO}. In order to verify that it is the appropriate generalisation to partial actions, we simply check that it is indeed a special case of Definition~\ref{def:TopAmenDefGr}.

\begin{proposition}\label{prop:AmenIffAmen}
A topological partial action $\theta \colon G \act \Omega$ of a discrete group on a locally compact Hausdorff space is topologically amenable in the sense of Definition~\ref{def:TopAmenDef} if and only if the transformation groupoid $\mathcal{G}_\theta$ is topologically amenable in the sense of Definition~\ref{def:TopAmenDefGr}.
\end{proposition}

\begin{proof}
Given a topological approximate invariant mean $(m_i)_{i \in I}$ on $\mathcal{G}_\theta$, define $\mu_i^x \in \text{Prob}(G)$ by $\mu_i^x(g):=m_i^x(g^{-1},\theta_g(x))$ for all $g \in G^x$, and set $\mu_i^x(g):=0$ for $g \notin G^x$. By assumption, $m_i^x$ is supported on the discrete set
$$\mathcal{G}_\theta^x=r^{-1}(x)=\{(g^{-1},\theta_g(x)) \in G \times \Omega \mid g \in G^x\},$$
which is mapped bijectively to $G^x$ under the map $(g^{-1},\theta_g(x)) \mapsto g$, so $\mu_i^x$ is well-defined. To check (1) for some given $i \in I$ and $g \in G$, fix $x_0 \in \Omega_{g^{-1}}$.  Then we may take a neighbourhood $x_0 \in U \subset \Omega_{g^{-1}}$ and a compactly supported continuous function $f \in C_c(\{g^{-1}\} \times \Omega_g) \subset C_c(\mathcal{G}_\theta)$ satisfying $0 \notin f(\{g^{-1}\} \times \theta_g(U))$. Now by assumption, the function
$$x \mapsto \int_{\mathcal{G}_\theta^{x}} f dm_i^x = \sum_{h \in G^x} f(h^{-1},\theta_h(x)) \cdot m_i^x(h^{-1},\theta_h(x))=f(g^{-1},\theta_g(x)) \cdot \mu_i^x(g)$$
is continuous, hence so is $U \ni x \mapsto \mu_i^x(g)$. For (2), observe first that
$$g.\mu_i^x(h)=\mu_i^x(hg)=m_i^x(g^{-1}h^{-1},\theta_{hg}(x))=(g,x) \cdot m_i^x\big(h^{-1},\theta_h(\theta_g(x)) \big)$$
for all $g \in G$, $x \in \Omega_{g^{-1}}$ and $h \in G^{\theta_g(x)}=G^x \cdot g^{-1}$. Given any compact set $K \subset \Omega_{g^{-1}}$, we therefore have
\begin{align*}
\sup_{x \in K} \big\Vert g.\mu_i^x - \mu_i^{\theta_g(x)} \big\Vert_1 &= 2 \cdot \sup_{x \in K} \big\Vert g.\mu_i^x - \mu_i^{\theta_g(x)} \big\Vert = 2 \cdot \sup_{\gamma \in \{g\} \times K}\big\Vert \gamma \cdot m_i^{s(\gamma)} - m_i^{r(\gamma)} \big\Vert \to 0
\end{align*}
as desired. We conclude that $(\mu_i)_{i \in I}$ is indeed a topological approximate invariant mean for $\theta$. The reverse implication is almost identical: Given a topological approximate invariant mean $(\mu_i)_{i \in I}$ for $\theta$, we set 
$$m_i^x(g,y):= \left\{\begin{array}{cl}
\mu_i^x(g^{-1}) & \If x=\theta_g(y) \\
0 & \text{otherwise}
\end{array} \right. $$
and observe that $m_i^x$ is indeed a probability measure supported on $\mathcal{G}_\theta^x$. Let $f \in C_c(\mathcal{G}_\theta)$ and observe that $F:=\big\{g \in G \mid f(\{g^{-1}\} \times \Omega_g) \ne \{0\} \big\}$ is a finite subset; we then have
$$\int_{\mathcal{G}_\theta^x} f dm_i^x = \sum_{g \in F \cap G^x} f(g^{-1},\theta_g(x))\mu_i^x(g)$$
for all $x \in \Omega=\mathcal{G}_\theta^{(0)}$. While every summand $x \mapsto  f(g^{-1},\theta_g(x))\mu_i^x(g)$ is continuous on $\Omega_{g^{-1}}$, the set $F \cap G^x$ need not vary continuously. To get around this, fix some $x_0 \in \Omega$ and set $F_{x_0}:=F \cap G^{x_0}$; we may of course assume that $F_{x_0} \subsetneq F$. From $f$ being compactly supported, there exist open neighbourhoods $U_g$ of $x_0$ for all $g \in F \setminus F_{x_0}$ such that either 
\begin{itemize}
\item $U_g \cap \Omega_{g^{-1}} = \emptyset$, or
\item $U_g \cap \Omega_{g^{-1}} \ne \emptyset$ and $f\big(\{g^{-1}\} \times\theta_g(U_g \cap \Omega_{g^{-1}}) \big) = \{0\}$.
\end{itemize}
Defining an open neighbourhood of $x_0$ by $U:=\bigcap_{g \in F \setminus F_{x_0}} U_g$, we then have
$$\int_{\mathcal{G}_\theta^x} f dm_i^x = \sum_{g \in F_{x_0}}  f(g^{-1},\theta_g(x))\mu_i^x(g)$$
for all $x \in U$. Since $x \mapsto \mu_i^x(g^{-1})$ is continuous in $x_0$ for every $g \in F_{x_0}$, we conclude that $x \mapsto \int_{\mathcal{G}_\theta^x} f dm_i^x$ is indeed continuous. (2) follows just as above by noting that every compact set $\mathcal{K} \subset \mathcal{G}_\theta$ is of the form $\bigsqcup_{i=1}^n \{g_i\} \times K_i$ for group elements $g_1,\ldots,g_n$ and compact subsets $K_i \subset \Omega_{g_i^{-1}}$.
\end{proof}

The real reason for our digression to groupoids is that one can associate a \textit{full} and a \textit{reduced} groupoid $C^*$-algebra $C^*(\mathcal{G})$ and $C^*_r(\mathcal{G})$ to any {\'e}tale locally compact Hausdorff groupoid, and in the case $\mathcal{G}=\mathcal{G}_\theta$, there are canonical isomorphisms
$$C^*(\mathcal{G}_\theta) \cong C_0(\Omega) \rtimes_\theta G \andspace C^*_r(\mathcal{G}_\theta) \cong C_0(\Omega) \rtimes_{\theta,r} G,$$
proven in \cite[Theorem 3.3 (2)]{Abadie2} and \cite[Proposition 2.2]{Li}. We thus obtain the following for free.
\begin{theorem}\label{thm:TopAmenThm}
Consider a partial action $\theta \colon G \act \Omega$ of a discrete group on a locally compact Hausdorff space. Then $C_0(\Omega) \rtimes_r G$ is nuclear if and only if $\theta$ is topologically amenable in the sense of Definition~\ref{def:TopAmenDef}, and in that case, the regular representation 
$$C_0(\Omega) \rtimes G \to C_0(\Omega) \rtimes_r G$$
is an isomorphism. In particular, $C_0(\Omega) \rtimes G$ is nuclear if and only if $C_0(\Omega) \rtimes_r G$ is nuclear.
\end{theorem}
\begin{proof}
In view of Remark~\ref{rem:BODef} and Proposition~\ref{prop:AmenIffAmen}, the first part follows immediately from \cite[Theorem 5.6.18 and Corollary 5.6.17]{BO}. Finally, if the full crossed product is nuclear, then so is the reduced, since nuclearity passes to quotients.
\end{proof}
In the setting of global actions, topological amenability can always be pulled back by continuous equivariant maps. In the partial setting, however, one has to be a little more careful.

\begin{definition}[{\cite[Definition 2.2]{ES}}]
Suppose that $G$ acts partially on $\Omega$ and $\Omega'$ and that $f \colon \Omega \to \Omega'$ is equivariant, so that $G^x \subset G^{f(x)}$ for all $x \in \Omega$. Then $f$ is called \textit{d-bijective} if $G^{f(x)}=G^x$ for all $x \in \Omega$.
\end{definition}
\begin{proposition}[{\cite[Proposition 2.4]{ES}}]\label{prop:DomainBij}
Assume that $\theta \colon G \act \Omega$ and $\theta' \colon G \act \Omega'$ are partial actions of a discrete group on locally compact Hausdorff spaces, and that $f \colon \Omega \to \Omega'$ is a continuous, equivariant and d-bijective map. If $\theta'$ is topologically amenable, then so is $\theta$.
\end{proposition} 
\begin{proof}
Let $(m_i)_{i \in I}$ denote a topological approximate invariant mean for $\theta'$ and define $\mu_i^x:=m_i^{f(x)}$ for all $x \in \Omega$ and $i \in I$. First observe that each $\mu_i^x$ is a probability measure on $G^x$ since $f$ is d-bijective, and that $\Omega_{g^{-1}} \ni x \mapsto \mu_i^x(g)$ is continuous, being the composition of $f$ and $\Omega_{g^{-1}}' \ni y \mapsto m_i^y(g)$. Finally, if $K \subset \Omega_{g^{-1}}$ is compact, then
$$\sup_{x \in K} \Vert g.\mu_i^x - \mu_i^{\theta_g(x)} \Vert_1 = \sup_{x \in K} \Vert g.m_i^{f(x)} - m_i^{f(\theta_g(x))} \Vert_1 =  \sup_{y \in f(K)} \Vert g.m_i^{y} - m_i^{\theta_g'(y)} \Vert_1 \to 0,$$
hence $\theta$ is indeed topologically amenable.
\end{proof}

One particular simple situation giving rise to a topological approximate invariant mean is the existence of a \textit{topological F{\o}lner net} for the action.
\begin{definition}\label{def:FolnerDef}
Let $\theta \colon G \act \Omega$ denote a partial action of a discrete group on a locally compact Hausdorff space, and denote by $\mathcal{F}(G)$ the set of non-empty finite subsets of $G$ endowed with the discrete topology. A \textit{topological F{\o}lner net for} $\theta$ is a net $(F_i)_{i \in I}$ of continuous (i.e.~locally constant) functions $F_i \colon \Omega \to \mathcal{F}(G)$, $x \mapsto F_i^x$, such that $F_i^x \subset G^x$ for all $x \in \Omega$, and for every $g \in G$ and all compact subsets $K \subset \Omega_{g^{-1}}$,
$$\sup_{x \in K} \frac{\vert F_i^x \cdot g^{-1} \setminus F_i^{\theta_g(x)} \vert}{\vert F_i^x \vert} \to 0.$$
\end{definition}\exend\medskip

Here is the observation that justifies Definition~\ref{def:FolnerDef}.
\begin{proposition}\label{prop:FolnerImpliesTopAmen}
Let $\theta \colon G \act \Omega$ denote a partial action of a discrete group on a locally compact Hausdorff space. If $\theta$ has a topological F{\o}lner net, then it is topologically amenable.
\end{proposition}
\begin{proof}
Assume that $(F_i)_{i \in I}$ is a topological F{\o}lner net for $\theta$, and define 
$$m_i \colon \Omega \to \text{Prob}(G) \quad \text{by} \quad m_i^x := \frac{1}{\vert F_i^x \vert} \cdot 1_{F_i^x},$$
where $1_F$ is the characteristic function on a set $F$. Then each $m_i^x$ is certainly a probability measure with support $F_i^x \subset G^x$, satisfying (1) of Definition~\ref{def:TopAmenDef}. In order to check (2), let $g \in G$ and compact $K \subset \Omega_{g^{-1}}$ be given. We then have
\begin{align*}
\Vert g.m_i^x - m_i^{\theta_g(x)} \Vert_1 &= \frac{\vert F_i^x \cdot g^{-1} \setminus F_i^{\theta_g(x)} \vert}{\vert F_i^x \vert} + \frac{\vert F_i^{\theta_g(x)} \setminus F_i^x \cdot g^{-1} \vert}{\vert F_i^{\theta_g(x)} \vert} \\
&{} \quad + \vert F_i^x \cdot g^{-1} \cap F_i^{\theta_g(x)} \vert \cdot \big\vert \frac{1}{\vert F_i^x \cdot g^{-1} \vert}-\frac{1}{\vert F_i^{\theta_g(x)}\vert}  \big\vert \\
&= \frac{\vert F_i^x \cdot g^{-1} \setminus F_i^{\theta_g(x)} \vert}{\vert F_i^x \vert} + \frac{\vert F_i^{\theta_g(x)} \cdot g \setminus F_i^{\theta_{g^{-1}}(\theta_g(x))} \vert}{\vert F_i^{\theta_g(x)} \vert} \\
&{} \quad + \vert F_i^x \cdot g^{-1} \cap F_i^{\theta_g(x)} \vert \cdot \frac{\big \vert \vert F_i^{\theta_g(x)} \vert - \vert F_i^x \cdot g^{-1} \vert \big\vert}{\vert F_i^x \vert \cdot \vert F_i^{\theta_g(x)} \vert} \\
&\le 2\cdot \frac{\vert F_i^x \cdot g^{-1} \setminus F_i^{\theta_g(x)} \vert}{\vert F_i^x \vert} + 2 \cdot \frac{\vert F_i^{\theta_g(x)} \cdot g \setminus F_i^{\theta_{g^{-1}}(\theta_g(x))} \vert}{\vert F_i^{\theta_g(x)} \vert}
\end{align*}
for all $x \in K$. Setting $K' := \theta_g(K) \subset \Omega_g$, we deduce that
$$\sup_{x \in K}\Vert g.m_i^x - m_i^{\theta_g(x)} \Vert_1 \le 2 \cdot \sup_{x \in K}\frac{\vert F_i^x \cdot g^{-1} \setminus F_i^{\theta_g(x)} \vert}{\vert F_i^x \vert} + 2 \cdot \sup_{y \in K'} \frac{\vert F_i^y \cdot g \setminus F_i^{\theta_{g^{-1}}(y)} \vert}{\vert F_i^y \vert} \to 0.$$
\end{proof}

Given a partial action $\theta \colon G \act \Omega$ with certain properties, one might desire a global action with the same properties. In terms of purely dynamical properties, this can always be accomplished by considering the \textit{minimal globalisation} $\widetilde{\theta} \colon G \act \widetilde{\Omega}$ of \cite[Theorem 2.5]{Abadie}. While the space $\widetilde{\Omega}$ resembles $\Omega$ locally, it might have very different global properties. The Hausdorff property might not even pass from $\Omega$ to $\widetilde{\Omega}$ \cite[Example 2.9]{Abadie}, but if the domains are clopen, such pathological examples do not exist \cite[Proposition 2.10]{Abadie}. However, compactness will \textit{usually} not be preserved, and so it is natural to ask if there exist ``good'' compactifications of $\widetilde{\theta}$. One particular case of interest is that of a topologically amenable partial action $\theta \colon G \act \Omega$ on a compact Hausdorff space with clopen domains. By the above, the minimal globalisation is topologically amenable and acts on a locally compact Hausdorff space. A good compactification of $\widetilde{\theta}$ in this context would be a topologically amenable one, and so the one-point compactification is not desirable if $G$ is non-amenable. Below, we will provide a good compactification in this setup for \textit{right-convex} partial actions of a free group.

\begin{definition}
A partial action $\mathbb{F} \act \Omega$ of a free group is called \textit{convex} if
$$\mathbb{F}^x=\{\alpha \in \mathbb{F} \mid x \in \Omega_{\alpha^{-1}}\}$$
is a right-convex subset of $\mathbb{F}$ for all $x \in \Omega$. \exend
\end{definition}

We will need the following technical lemma.

\begin{lemma}\label{lem:CompHausd}
Consider a continuous, surjective map of Hausdorff spaces $p \colon \Omega \to \Upsilon$, and assume that for any $y \in \Upsilon$, there exists an open neighbourhood $U$ of $y$ for which $\overline{p^{-1}(U)}$ is compact. If $C$ is any compactification of $\Upsilon$, then $\Omega \sqcup \partial \Upsilon$ is a Hausdorff compactification of $\Omega$ when equipped with the smallest topology making $\Omega \hookrightarrow \Omega \sqcup \partial \Upsilon$ open and $p \sqcup \id \colon \Omega \sqcup \partial \Upsilon \to C$ continuous.
\end{lemma}

\begin{proof}
We first show that $\Omega \sqcup \partial \Upsilon$ is Hausdorff, so consider any pair of distinct points $x_1,x_2 \in \Omega \sqcup \partial \Upsilon$. Since $\Omega$ is open in $\Omega \sqcup \partial \Upsilon$, we may assume that at least one of the points belongs to $\partial \Upsilon$. Then $p \sqcup \id(x_1) \ne p \sqcup \id(x_2)$, so they can be separated by $\Upsilon \sqcup \partial \Upsilon$ being Hausdorff and continuity of $p \sqcup \id$. In order to demonstrate compactness, take any net $(x_i)_{i \in I}$ in $\Omega \sqcup \partial \Upsilon$ and consider the net $(y_i)_{i \in I}$ with $y_i=p \sqcup \id (x_i)$. By compactness of $C$, it has a convergent subnet $(y_i)_{i \in J}$ and we denote the limit point by $y$. If $y \in \Upsilon$, then there exist an open neighbourhood $U$ of $y$ for which $\overline{p^{-1}(U)}$ is compact, and since $\overline{p^{-1}(U)}$ contains a subnet of $(x_i)_{i \in I}$, it has a convergent subnet. Assume instead that $y \in \partial \Upsilon$. Then, since every open neighbourhood of $y$ inside $\Omega \sqcup \partial \Upsilon$ is of the form $(p \sqcup \id)^{-1}(U)$ for an open neighbourhood $y \in U \subset C$, we see that $(x_i)_{i \in J}$ converges towards $y$.
\end{proof}

\begin{theorem}\label{thm:Comp}
Suppose that $\theta \colon \mathbb{F} \act \Omega$ is a convex partial action of a free group of rank at least two on a compact Hausdorff space with clopen domains. Then there exists a Hausdorff compactification $\widehat{\Omega}$ of $\widetilde{\Omega}$ and an extension $\widehat{\theta} \colon \mathbb{F} \act \widehat{\Omega}$ of $\widetilde{\theta}$, such that the restriction $\widehat{\theta} \colon \mathbb{F} \act \widehat{\Omega} \setminus \widetilde{\Omega}$ is conjugate to the canonical boundary action $\mathbb{F} \act \partial \mathbb{F}$. In particular, if $\theta$ is topologically amenable, then so is $\widehat{\theta}$.
\end{theorem}

\begin{proof}
Recall that
$$\widetilde{\Omega}=\frac{\mathbb{F} \times \Omega}{\sim}, \text{\quad where $(\alpha,x) \sim (\beta,y) \Leftrightarrow x \in \Omega_{\alpha^{-1}\beta}$ and $y=\theta_{\beta^{-1}\alpha}(x)$}.$$
The action is simply induced from the group, $\widetilde{\theta}_\beta([\alpha,x])=[\beta\alpha,x]$, and $\Omega$ embeds as a clopen subspace by the map $\iota \colon \Omega \to \widetilde{\Omega}$, $\iota(x)=[1,x]$. Now since $\theta$ is assumed convex, to each pair $(\alpha,x) \in \mathbb{F} \times \Omega$ there is unique minimal word $\sigma(\alpha,x) \in \mathbb{F}$ such that $\sigma(\alpha,x)^{-1} \cdot \alpha \in \mathbb{F}^x$, and it is straightforward to check that $\sigma$ respects the relation $\sim$. We claim that $\sigma$ is also continuous, so consider any pair $(\alpha,x)$. Assuming first that $\sigma(\alpha,x)=1$, we note that $\{\alpha\} \times \Omega_{\alpha^{-1}}$ is an open neighbourhood of $(\alpha,x)$ on which $\sigma$ attains the value $1$. If $\sigma(\alpha,x) \ne 1$, we write $\alpha_x$ for the maximal subword (read from the right) of $\alpha$ satisfying $\alpha_x \in \mathbb{F}^x$ and denote the following letter by $s$. Then $\sigma$ is constant on the open neighbourhood 
$$\{\alpha\} \times \big( \Omega_{\alpha_x^{-1}} \setminus \Omega_{(s\alpha_x)^{-1}} \big),$$
where we invoke the assumption of clopen domains, hence it is indeed continuous. In conclusion, $\sigma$ drops to a continuous map $\sigma \colon \widetilde{\Omega} \to \mathbb{F}$. Also observe that if $\beta \in \mathbb{F}$ does not contain $\sigma([\alpha,x])^{-1}$ as a subword (read from the right), then
$$\sigma(\widetilde{\theta}_\beta ([\alpha,x]))=\sigma([\beta \cdot \alpha,x])=\beta \cdot \sigma([\alpha,x]).$$
Now define $\widehat{\Omega}:=\widetilde{\Omega} \cup \partial \mathbb{F}$ as a set and equip it with the smallest topology that makes the inclusion $\widetilde{\Omega} \hookrightarrow \widehat{\Omega}$ open and the map $\sigma \cup \id \colon \widehat{\Omega} \to \mathbb{F} \cup \partial \mathbb{F}$ continuous. By Lemma~\ref{lem:CompHausd}, this makes $\widehat{\Omega}$ into a compact Hausdorff space, and we extend the action by the ordinary action of $\mathbb{F}$ on its boundary. It follows immediately from the above observation that the action on $\widetilde{\Omega}$ is compatible with that on the boundary, and so we obtain a short exact sequence
$$0 \rightarrow C_0(\widetilde{\Omega}) \rtimes \mathbb{F} \rightarrow C(\widehat{\Omega}) \rtimes \mathbb{F} \rightarrow C(\partial \mathbb{F}) \rtimes \mathbb{F} \rightarrow 0.$$
Since both the ideal and the quotient are nuclear, we conclude that the extension is nuclear as well, hence $\widehat{\theta}$ is topologically amenable.
\end{proof}

\section{Condition (N) and proper orientability of the branching subgraph}
In this section, we introduce Condition (N) for finitely separated graphs and prove that it is equivalent to both exactness and nuclearity of the $C^*$-algebra associated to a certain subgraph. But first, we introduce quite a bit of terminology.
\begin{definition}\label{def:LocOr}
A non-trivial admissible path $\alpha$ in a separated graph $(E,C)$ is said to \textit{allow a return} if there is an admissible path $\beta$ making $\beta\alpha$ a closed path, and a set $X \in C_v$ then allows a return if $e^{-1}$ allows a return for some $e \in X$. A vertex $v \in E^0$ is called a \textit{branching vertex} if
$$\vert \{e \in s^{-1}(v) \mid e \text{ allows a return}\} \vert + \vert \{X \in C_v \mid X \text{ allows a return}\} \vert \ge 3,$$
and a branching vertex $v$ is said to \textit{admit a local orientation} if one of the following holds:
\begin{enumerate}
\item There exists $X_v \in C_v$ such that for every base-simple closed path $\alpha$ at $v$, either $\ini_d(\alpha) \in X_v^{-1}$ or $\ter_d(\alpha) \in X_v$.
\item There is an edge $e_v \in s^{-1}(v)$ such that for every base-simple closed path $\alpha$ at $v$, either $\ini_d(\alpha)=e_v$ or $\ter_d(\alpha)=e_v^{-1}$.
\end{enumerate}
Observe that if $v$ is a branching vertex satisfying (1), then it does not satisfy (2) and $X_v$ is unique. Likewise, if $v$ is a branching vertex satisfying (2), then it does not satisfy (1) and $e_v$ is unique. We will therefore refer to the branching vertices admitting local orientations as either \textit{type} (1) or (2).
\end{definition}

\begin{lemma}
If $v$ is a branching vertex admitting a local orientation, then it satisfies either Definition~\ref{def:LocOr}\textup{(}1\textup{)} or Definition~\ref{def:LocOr}\textup{(}2\textup{)} for arbitrary closed paths $\alpha$ based at $v$.
\end{lemma}
\begin{proof}
We verify the claim for Definition~\ref{def:LocOr}(1); the proof in the second case is exactly the same. If $\alpha$ is an arbitrary closed path, we can decompose into base-simple closed paths $\alpha=\alpha_n \cdots \alpha_1$. Assuming that the claim holds for all products of $n-1$ base-simple closed paths, we either have $\ini_d(\alpha)=\ini_d(\alpha_{n-1}\cdots \alpha_1) \in X_v^{-1}$, in which case we are done, or $\ter_d(\alpha_{n-1}) \in X_v$. In the latter case, we must have $\ini_d(\alpha_n) \not\in X_v^{-1}$ so $\ter_d(\alpha)=\ter_d(\alpha_n) \in X_v$.
\end{proof}

\begin{remark}
If $v$ admits a local orientation, one should regard $X_v^{-1}$ or $e_v$ (depending on the type of $v$) as the proper exits of $v$, while the rest of $s^{-1}(v) \sqcup r^{-1}(v)^{-1}$ should be regarded as entries. Then any cycle based at $v$ will depart from $v$ using one and arrive at $v$ using the other, so there is a canonical orientation of the cycle, i.e.~a canonical choice between itself and its inverse. This is why we call it a local orientation. \exend
\end{remark}

We now have the language to define Condition (N).

\begin{definition}
A finitely separated graph $(E,C)$ is said to satisfy \textit{Condition} (\textit{N}) if any branching vertex $v \in E^0$ admits a local orientation. \exend
\end{definition}

It is worth noting that Condition (N) trivially passes to subgraphs.

\begin{cexample}
Below, we have marked the branching vertices in blue:
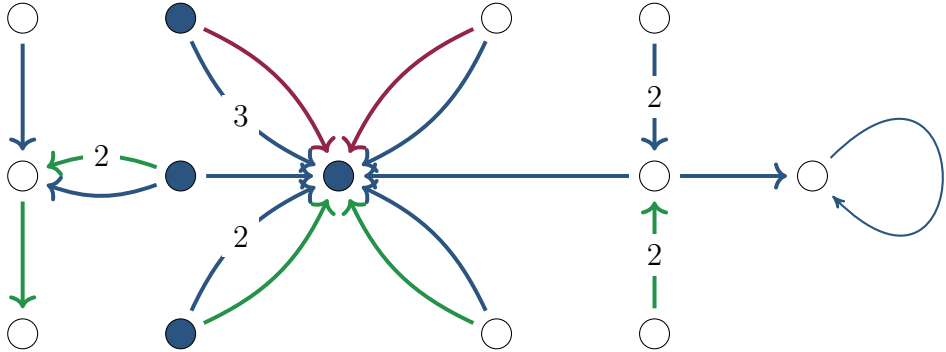
\begin{center}
\begin{tikzpicture}[scale=0.70]
 \SetUpEdge[lw         = 1.5pt,
            labelcolor = white]
  \tikzset{VertexStyle/.style = {draw,shape = circle,fill = white, inner sep=4pt,minimum size=10pt,outer sep=4pt}}

  \SetVertexNoLabel

  \Vertex[x=-6,y=0]{u1}  
  \Vertex[x=6,y=0]{u4}
  \Vertex[x=9,y=0]{u5}
  
  \Vertex[x=3,y=3]{u8}
  \Vertex[x=-6,y=3]{u6}
  \Vertex[x=6,y=3]{u9}

  \Vertex[x=-6,y=-3]{u10}
  \Vertex[x=3,y=-3]{u12}
  \Vertex[x=6,y=-3]{u13}
  
  \tikzset{VertexStyle/.style = {draw,shape = circle,fill = \niceblue, inner sep=4pt,minimum size=10pt,outer sep=4pt}}
  \Vertex[x=-3,y=0]{u2}
  \Vertex[x=0,y=0]{u3}
  \Vertex[x=-3,y=3]{u7}
  \Vertex[x=-3,y=-3]{u11}

  \tikzset{EdgeStyle/.style = {->,color={\niceblue}}}
  \Loop[dir=EA](u5)  
  \Edge[](u2)(u3)
  \Edge[](u4)(u3)
  \Edge[](u4)(u5)
  \Edge[](u6)(u1)
  \Edge[label=$2$](u9)(u4)
  
  \tikzset{EdgeStyle/.style = {->,bend left=20,color={\niceblue}}}
  \Edge[](u2)(u1)
  \Edge[](u8)(u3)
  \Edge[label=$2$](u11)(u3)
  
  \tikzset{EdgeStyle/.style = {->,bend right=20,color={\niceblue}}}
  \Edge[label=$3$](u7)(u3)
  \Edge[](u12)(u3)
  
  \tikzset{EdgeStyle/.style = {->,color={\nicegreen}}}
  \Edge[label=$2$](u13)(u4)
  \Edge[](u1)(u10)
  
  \tikzset{EdgeStyle/.style = {->,bend left=20,color={\nicegreen}}}
  \Edge[](u12)(u3)
  
  \tikzset{EdgeStyle/.style = {->,bend right=20,color={\nicegreen}}}
  \Edge[](u11)(u3)
  \Edge[label=$2$](u2)(u1)
  
  \tikzset{EdgeStyle/.style = {->,bend left=20,color={\nicered}}}
  \Edge[](u7)(u3)
  
  \tikzset{EdgeStyle/.style = {->,bend right=20,color={\nicered}}}
  \Edge[](u8)(u3)
\end{tikzpicture}
\end{center}
The reader may check that this graph actually satisfies Condition (N) as the right branching vertex is of type (1), and the left ones are of type (2). \exend
\end{cexample}

Rather than simply negating the above definition, we would also like to have a constructive understanding of what it means for a graph not to satisfy Condition (N). For clarity, we first introduce the following technical lemma.

\begin{lemma}\label{lem:CoolCombs}
Let $P$ and $S$ denote sets with the following structure: There are functions $\iota,\tau \colon P \to S$, an associative partial composition
$$\{(\alpha_2,\alpha_1) \in P^2 \mid \iota(\alpha_2) \ne \tau(\alpha_1) \} \to P \quad , \quad (\alpha_2,\alpha_1) \mapsto \alpha_2\alpha_1,$$
and a function $P \ni \alpha \mapsto \alpha^{-1} \in P$, such that
$$\iota(\alpha_2\alpha_1)=\iota(\alpha_1), \quad \tau(\alpha_2\alpha_1)=\tau(\alpha_2), \quad \iota(\alpha^{-1})=\tau(\alpha), \andspace \tau(\alpha^{-1})=\iota(\alpha).$$
Moreover, assume that for all $\alpha_1,\alpha_2,\alpha_3 \in P$,
$$\vert \iota(\{\alpha_1,\alpha_2,\alpha_3\}) \vert \le 2 \quad \text{or} \quad \vert \tau(\{\alpha_1,\alpha_2,\alpha_3\}) \vert \le 2.$$
If $\vert \iota(P) \vert = \vert \tau(P) \vert \ge 3$, then there is a unique $s \in S$ satisfying $s \in \{\iota(\alpha),\tau(\alpha)\}$ for all $\alpha \in P$.
\end{lemma}

\begin{proof}
Take $\alpha_1,\alpha_2,\alpha_3 \in P$ with distinct $\iota(\alpha_i)$'s and define $\mathcal{E}(\alpha):=\{\iota(\alpha),\tau(\alpha)\}$ for all $\alpha \in P$. We first claim that $\mathcal{E}(\alpha_i) \cap \mathcal{E}(\alpha_j) \ne \emptyset$ for all $i,j$, and without loss of generality we may take $i=1$ and $j=2$. Assume in order to reach a contradiction that $\mathcal{E}(\alpha_1) \cap \mathcal{E}(\alpha_2) = \emptyset$. If $\iota(\alpha_1) \ne \tau(\alpha_1)$, then 
$$\vert \iota(\{\alpha_1,\alpha_1^{-1},\alpha_2\}) \vert = \vert \tau(\{\alpha_1,\alpha_1^{-1},\alpha_2\}) \vert = 3,$$
hence $\iota(\alpha_1) = \tau(\alpha_1)$, and similarly we must have $\iota(\alpha_2) = \tau(\alpha_2)$. We deduce that either 
$$\mathcal{E}(\alpha_1) \cap \mathcal{E}(\alpha_3) = \emptyset \quad \text{or} \quad \mathcal{E}(\alpha_2) \cap \mathcal{E}(\alpha_3) = \emptyset,$$
and without loss of generality we may assume the former. But then
$$\vert \iota(\{\alpha_1,\alpha_2,\alpha_3^{-1}\alpha_1\alpha_3\}) \vert = \vert \tau(\{\alpha_1,\alpha_2,\alpha_3^{-1}\alpha_1\alpha_3\}) \vert = 3,$$
giving us our desired contradiction. We now even claim that 
$$\mathcal{E}(\alpha_1) \cap \mathcal{E}(\alpha_2) \cap \mathcal{E}(\alpha_3) \ne \emptyset.$$
Assuming the contrary, we can arrange that
$$\tau(\alpha_1)=\iota(\alpha_2), \quad \tau(\alpha_2)=\iota(\alpha_3), \andspace \tau(\alpha_3) = \iota(\alpha_1)$$
by applying the first part and possibly interchanging the indices and taking inverses, hence
$$\vert \iota(\{\alpha_1,\alpha_2,\alpha_3\}) \vert = \vert \tau(\{\alpha_1,\alpha_2,\alpha_3\}) \vert = 3.$$
We conclude that 
$$\mathcal{E}(\alpha_1) \cap \mathcal{E}(\alpha_2) \cap \mathcal{E}(\alpha_3) = \{s\}$$
for some $s \in S$. This implies that there are distinct $s_1,s_2 \ne s$ and $i \ne j$ such that $\mathcal{E}(\alpha_i) = \{s,s_1\}$ and $\mathcal{E}(\alpha_j)=\{s,s_2\}$. Now if $\alpha \in P$ is arbitrary, then by taking suitable inverses $\beta=\alpha^\varepsilon$, $\beta_i=\alpha_i^{\varepsilon_i}$, and $\beta_j = \alpha_j^{\varepsilon_j}$ (i.e.~$\varepsilon,\varepsilon_i,\varepsilon_j \in \{-1,1\}$), we can arrange that
$$\vert \iota(\{\beta,\beta_i,\beta_j\}) \vert = 3.$$
This allows us to apply the above conclusions, hence
$$\emptyset \ne \mathcal{E}(\beta) \cap \mathcal{E}(\beta_i) \cap \mathcal{E}(\beta_j)  \subset \mathcal{E}(\alpha_i) \cap \mathcal{E}(\alpha_j) =\{s\}.$$
We deduce that $\mathcal{E}(\beta) \cap \mathcal{E}(\beta_i) \cap \mathcal{E}(\beta_j) = \{s\}$, so in particular $s \in \mathcal{E}(\beta)=\mathcal{E}(\alpha)$. Uniqueness of $s$ is clear from $\mathcal{E}(\alpha_i) \cap \mathcal{E}(\alpha_j) =\{s\}$.
\end{proof}

\begin{proposition}\label{prop:NonCondN}
A finitely separated graph $(E,C)$ does not satisfy Condition \textup{(}N\textup{)} if and only if there is a branching vertex $v \in E^0$ and cycles $\alpha=\delta\gamma$ and $\beta=\varepsilon\gamma$ based at $v$ with $\gamma=\alpha \wedge \beta<\alpha,\beta$, such that $\beta \cdot \alpha^{-1}=\varepsilon\delta^{-1}$ and $\beta\alpha$ are cycles.
\end{proposition}

\begin{proof}
It is clear that if such $\alpha$ and $\beta$ exist, then $v$ does not admit a local orientation. Now let $v$ denote any branching vertex in $(E,C)$ and assume instead that such $\alpha$ and $\beta$ do not exist. Letting $\pi$ denote the map $s^{-1}(v) \sqcup r^{-1}(v)^{-1} \to s^{-1}(v) \sqcup C_v$ given by $\pi(e):=e$ and $\pi(e^{-1}):=[e]$, we define sets $P:=\{\text{closed paths based at $v$}\}$ and $S:=s^{-1}(v) \sqcup C_v$ along with maps $\iota,\tau \colon P \to S$ given by $\iota(\alpha):=\pi(\ini_d(\alpha))$ and $\tau(\alpha):=\pi(\ter_d(\alpha)^{-1})$. Obviously, given two closed paths $\alpha_1,\alpha_2 \in P$, the concatenated product $\alpha_2\alpha_1$ is in $P$ if and only if $\iota(\alpha_2) \ne \tau(\alpha_1)$. In fact, the only assumption of Lemma~\ref{lem:CoolCombs} which is not obviously satisfied is that either $\vert \iota(\{\alpha_1,\alpha_2,\alpha_3\}) \vert \le 2$ or $\vert \tau(\{\alpha_1,\alpha_2,\alpha_3\}) \vert \le 2$ for all triples $\alpha_1,\alpha_2,\alpha_3 \in P$. Assume in order to reach a contradiction that
$$\vert \iota(\{\alpha_1,\alpha_2,\alpha_3\}) \vert = \vert \tau(\{\alpha_1,\alpha_2,\alpha_3\}) \vert = 3$$
for one such triple. Then $\alpha:=\alpha_2^{-1}\alpha_1$ and $\beta:=\alpha_3^{-1}\alpha_1$ are cycles with $\alpha \wedge \beta=\alpha_1< \alpha,\beta$ such that both $\beta \cdot \alpha^{-1}=\alpha_3^{-1}\alpha_2$ and $\beta\alpha=\alpha_3^{-1}\alpha_1\alpha_2^{-1}\alpha_1$ are cycles as well, contradicting our assumption. It now follows immediately from Lemma~\ref{lem:CoolCombs} that $v$ admits a local orientation.
\end{proof}

With the above characterisation at hand, we can already prove that Condition (N) is a necessary condition for exactness of $\mathcal{O}(E,C)$.

\begin{proposition}\label{prop:Exactness}
Let $(E,C)$ denote a finitely separated graph and consider the statements
\begin{enumerate}
\item[\textup{(1)}] The $C^*$-algebra $\mathcal{O}(E,C)$ is exact.
\item[\textup{(2)}] Every stabiliser of the partial action $\theta^{(E,C)}$ is amenable \textup{(}hence trivial or cyclic\textup{)}.
\item[\textup{(3)}] $(E,C)$ satisfies Condition \textup{(}N\textup{)}.
\end{enumerate}
Then $(1) \Rightarrow (2) \Rightarrow (3)$.
\end{proposition}
\begin{proof}
Assuming that $(E,C)$ does not satisfy Condition (N), there are cycles $\alpha$ and $\beta$ as in Proposition~\ref{prop:NonCondN}. Observe that any reduced product of $\alpha$'s, $\beta$'s and their inverses is admissible, and denote by $F$ the free subgroup of $\mathbb{F}$ generated by $\alpha$ and $\beta$. If $F$ were of rank $1$, then we would have $\alpha=\gamma^m$ and $\beta=\gamma^n$ for some cycle $\gamma$ and non-zero integers $m,n$. As $\beta\alpha$ is admissible, we see that $m$ and $n$ have the same sign, contradicting $\alpha \wedge \beta < \alpha,\beta$. We conclude that $F \cong \mathbb{F}_2$. Now $\omega:=\langle F \rangle$ defines an $(E,C)$-animal, and we can find a configuration $\xi \in \Omega(E,C)_\omega$ with $F \le \text{Stab}(\xi)$. Formally, this construction can be carried out as follows: Take any $\eta \in \Omega(E,C)$ with $\{\alpha,\beta,\ter_d(\alpha)^{-1},\ter_d(\beta)^{-1}\} \subset \eta$ and consider the animal
$$\chi:=\{\gamma \in \eta \mid \gamma \not\ge \alpha,\beta,\ter_d(\alpha)^{-1},\ter_d(\beta)^{-1}\}.$$
Then one may verify that $\xi:=\bigsqcup_{\sigma \in F} \chi \cdot \sigma$ defines a configuration, and by construction $F \le \text{Stab}(\xi)$, so (2) does not hold. It finally follows that $\mathcal{O}(E,C)$ is non-exact by \cite[Proposition 7.1(i)]{AEK}.
\end{proof}

The aim of the rest of this paper is to prove that Condition (N) in fact implies nuclearity of $\mathcal{O}(E,C)$. Roughly speaking, the idea is to decompose the graph into one part 'spanned' by the branching vertices, and a complementary part containing no branching vertices, and then deal with these two subgraphs separately. In the remainder of this section, we will treat the former graph.

\begin{definition}
Given any finitely separated graph $(E,C)$, define a relation on $E^0$ by $u \multimap v$ if there is an admissible path $\alpha \colon u \to v$ and a cycle $\beta$ based at $v$, such that $\alpha^{-1}\beta\alpha$ is admissible. Note that $\multimap$ is transitive, but in general it is not reflexive, symmetric or antisymmetric. In fact, $u \multimap u$ if and only if $u$ admits a cycle as we may take $\alpha=u$.
\end{definition}

\begin{definition}
Let $(E,C)$ denote a finitely separated graph. The \textit{branching subgraph} $(E_{\textup{Br}},C^{\textup{Br}})$ is the full subgraph with vertex set 
$$E_{\textup{Br}}^0 := \big\{ u \in E^0 \mid u \multimap v \text{ for a branching vertex } v\big\}.$$
\end{definition}

\begin{remark} 
Note that if there is a closed path based at $u$ passing through a branching vertex $v$, then automatically $u \multimap v$.
\end{remark}

\begin{cexample}
The branching subgraph of our example is as indicated below:
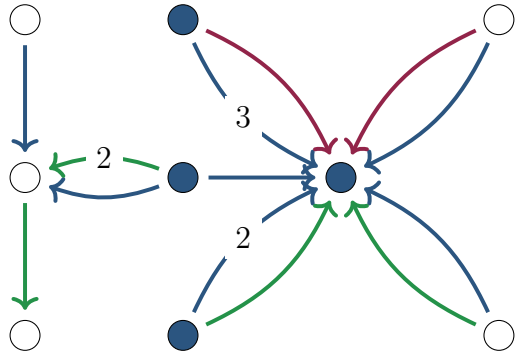
\begin{center}
\begin{tikzpicture}[scale=0.70]
 \SetUpEdge[lw         = 1.5pt,
            labelcolor = white]
  \tikzset{VertexStyle/.style = {draw,shape = circle,fill = white, inner sep=4pt,minimum size=10pt,outer sep=4pt}}

  \SetVertexNoLabel

  \Vertex[x=-6,y=0]{u1}  
    
  \Vertex[x=3,y=3]{u8}
  \Vertex[x=-6,y=3]{u6}

  \Vertex[x=-6,y=-3]{u10}
  \Vertex[x=3,y=-3]{u12}
  
  \tikzset{VertexStyle/.style = {draw,shape = circle,fill = \niceblue, inner sep=4pt,minimum size=10pt,outer sep=4pt}}
  \Vertex[x=-3,y=0]{u2}
  \Vertex[x=0,y=0]{u3}
  \Vertex[x=-3,y=3]{u7}
  \Vertex[x=-3,y=-3]{u11}

  \tikzset{EdgeStyle/.style = {->,color={\niceblue}}}
  \Edge[](u2)(u3)
  \Edge[](u6)(u1)
  
  \tikzset{EdgeStyle/.style = {->,bend left=20,color={\niceblue}}}
  \Edge[](u2)(u1)
  \Edge[](u8)(u3)
  \Edge[label=$2$](u11)(u3)
  
  \tikzset{EdgeStyle/.style = {->,bend right=20,color={\niceblue}}}
  \Edge[label=$3$](u7)(u3)
  \Edge[](u12)(u3)

  \tikzset{EdgeStyle/.style = {->,color={\nicegreen}}}
  \Edge[](u1)(u10)  
  
  \tikzset{EdgeStyle/.style = {->,bend left=20,color={\nicegreen}}}
  \Edge[](u12)(u3)
  
  \tikzset{EdgeStyle/.style = {->,bend right=20,color={\nicegreen}}}
  \Edge[](u11)(u3)
  \Edge[label=$2$](u2)(u1)
  
  \tikzset{EdgeStyle/.style = {->,bend left=20,color={\nicered}}}
  \Edge[](u7)(u3)
  
  \tikzset{EdgeStyle/.style = {->,bend right=20,color={\nicered}}}
  \Edge[](u8)(u3)
\end{tikzpicture}
\end{center}
\exend
\end{cexample}

Next, we introduce the notion of a (\textit{proper}) \textit{orientation}.

\begin{definition}\label{def:Or}
Let $(E,C)$ denote a finitely separated graph. A \textit{proper orientation} of $(E,C)$ is a decomposition $E^1=E^1_+ \sqcup E^1_-$ such that, for every $v \in E^0$, one of the following holds:
\begin{enumerate}
\item $E^1_- \cap r^{-1}(v) \in C_v$ and $E^1_+ \cap s^{-1}(v) = \emptyset$.
\item $E^1_- \cap r^{-1}(v) = \emptyset$ and $\vert E^1_+ \cap s^{-1}(v) \vert = 1$.
\end{enumerate}
If (2) is replaced by the weaker assumption
\begin{enumerate}
\item[(2')] $E^1_- \cap r^{-1}(v) = \emptyset$ and $\vert E^1_+ \cap s^{-1}(v) \vert \le 1,$
\end{enumerate}
then it will simply be called an \textit{orientation}. We shall often regard an orientation as a map $\mathfrak{o} \colon E^1 \to \{-1,1\}$, where
$$\mathfrak{o}(e)=\left\{
\begin{array}{cl}
1 & \If e \in E^1_+ \\
-1 & \If e \in E^1_- 
\end{array}
\right., $$
and as in \cite{AL}, an admissible path of the form
$$e_n^{\mathfrak{o}(e_n)}e_{n-1}^{\mathfrak{o}(e_{n-1})} \cdots e_2^{\mathfrak{o}(e_2)}e_1^{\mathfrak{o}(e_1)}$$
will be referred to as \textit{positively oriented}, while an admissible path of the form
$$e_n^{-\mathfrak{o}(e_n)}e_{n-1}^{-\mathfrak{o}(e_{n-1})} \cdots e_2^{-\mathfrak{o}(e_2)}e_1^{-\mathfrak{o}(e_1)}$$
is called \textit{negatively oriented}. \exend
\end{definition}

We first mention a pair of trivial, yet important, observations.
\begin{lemma}\label{LemmaOrient}
Assume that $(E,C)$ is oriented. Then every path $\alpha$ is of the form $\alpha=\alpha_{-}\alpha_{+}$, where $\alpha_{+}$ and $\alpha_{-}$ are \textup{(}possibly trivial\textup{)} positively and negatively oriented paths, respectively.
\end{lemma}
\begin{proof}
We simply have to check that if $f^\varepsilon e^{-\mathfrak{o}(e)}$ is admissible, then $\varepsilon=-\mathfrak{o}(f)$. Assume first that $\mathfrak{o}(e)=-1$; then $r(e)$ must satisfy Definition~\ref{def:Or}(1) since $E^1_- \cap r^{-1}(v) \ne \emptyset$.  Now if $\varepsilon=1$ so that $s(f)=r(e)$, then necessarily $\mathfrak{o}(f)=-1=-\varepsilon$. Conversely if $\varepsilon=-1$, then $r(f)=r(e)$ with $[e] \ne [f]$, so $\mathfrak{o}(f)=1=-\varepsilon$. The case $\mathfrak{o}(e)=1$ is completely similar.
\end{proof}
\begin{lemma}\label{lem:MaximalPosPath}
Assume that $(E,C)$ is properly oriented and let $\xi \in \Omega(E,C)$. Then for every $n \ge 1$, there is a unique positively oriented admissible path $\xi_n \in \xi$ of length $n$.
\end{lemma}
\begin{proof}
This is clear from Definition~\ref{def:Or}.
\end{proof}

We can now easily prove that properly oriented graphs give rise to topologically amenable actions.

\begin{proposition}
If $(E,C)$ is properly oriented, then the partial action $\theta^{(E,C)}$ admits a topological F{\o}lner sequence.
\end{proposition}

\begin{proof}
Define $F_n^\xi:=\{\xi_k \mid k \le n\}$ for all $\xi \in \Omega(E,C)$ and $n \ge 1$, using the notation of Lemma~\ref{lem:MaximalPosPath}. Given any $\alpha \in \mathbb{F}$ and $\xi \in \Omega(E,C)_{\alpha^{-1}}$, we write $\alpha=\alpha_-\alpha_+$ and, for the sake of notational simplicity, set $\eta:=\theta_{\alpha_+}(\xi)$. Then $\alpha_-^{-1}$ is positively oriented, so 
$$F_{\vert \alpha_+ \vert + n}^\xi=\{\beta \mid 1 < \beta \le \alpha_+\} \sqcup F_n^\eta \cdot \alpha_+ \andspace F_{\vert \alpha_- \vert + n}^{\theta_{\alpha}(\xi)} = \{\beta \mid 1<\beta \le \alpha_-^{-1}\} \sqcup F_n^\eta \cdot \alpha_-^{-1}$$
for all $n \ge 1$. We see that
$$F_{\vert \alpha_+ \vert + n}^\xi \cdot \alpha^{-1}= \{\beta \mid \alpha_-^{-1} \le \beta < \alpha^{-1}\} \sqcup F_n^\eta \cdot \alpha_-^{-1}$$
hence
$$\frac{\big\vert F_{\vert \alpha \vert + n} \cdot \alpha^{-1} \setminus F_{\vert \alpha \vert + n}^{\theta_\alpha(\xi)} \big\vert}{\big\vert F_{\vert \alpha \vert + n}^\xi \big\vert} \le \frac{\vert \alpha \vert}{\vert \alpha \vert + n}$$
with the upper bound converging to $0$ as $n \to \infty$ uniformly on $\Omega(E,C)_{\alpha^{-1}}$.
\end{proof}

In view of the above result, our goal is to extend the local orientations of a Condition (N) graph into a proper orientation of the entire branching subgraph. This requires some preparation, which we provide just below. First though, we need a bit more terminology.

\begin{definition}
A vertex $v$ is called \textit{weakly branching} if there is a cycle passing through both $v$ and a branching vertex, and an edge $e$ is called \textit{critical} if $e$ does not allow a return and $r(e)$ is weakly branching. 
\end{definition}

\begin{lemma}\label{lem:Types}
Let $(E,C)$ denote a finitely separated graph. Any edge $e \in E_{\textup{Br}}^1$ satisfies exactly one of the following:
\begin{enumerate}
\item[\textup{(1)}] $e$ is on a cycle passing through a branching vertex.
\item[\textup{(2)}] $e$ is critical.
\item[\textup{(3a)}] If $\alpha$ is a closed path in $(E_{\textup{Br}},C^{\textup{Br}})$, based at $r(e)$ and passing through a branching vertex, then $\ini_d(\alpha) \in [e]^{-1}$.
\item[\textup{(3b)}] If $\alpha$ is a closed path in $(E_{\textup{Br}},C^{\textup{Br}})$, based at $s(e)$ and passing through a branching vertex, then $\ini_d(\alpha)=e$.
\end{enumerate}
\end{lemma}

\begin{proof}
It is clear that the above statements are mutually exclusive. In order to see that they cover all edges, assume that $e \in E_{\textup{Br}}^1$ does not satisfy (2), (3b) or (3a). Then we can take closed paths $\alpha_1$, $\alpha_2$ in $(E_{\textup{Br}},C^{\textup{Br}})$ passing through branching vertices, such that $s(\alpha_1)=s(e)$, $s(\alpha_2)=r(e)$, $\ini_d(\alpha_1) \ne e$ and $\ini_d(\alpha_2) \notin [e]^{-1}$. We may assume that $\ter_d(\alpha_1) \ne e^{-1}$, so $e\alpha_1e^{-1}$ is admissible as well as $\ter_d(\alpha_2) \ne e$. Observe that if $e^{-1} \alpha_2$ is admissible, then $\alpha_1 \alpha_1 e^{-1} \alpha_2 e$ is a cycle passing through a branching vertex, so let us assume it is not. Then $\beta_2$ is a cycle, so $r(e)$ is weakly branching. It follows from $e$ being non-critical that $e$ admits a return, and we let $\beta$ denote an admissible path for which $\beta e$ is a closed path. Finally, if $\ter_d(\beta) \ne e^{-1}$, then $\beta \alpha_2 e$ is a cycle, and if $\ter_d(\beta)=e^{-1}$, then $\alpha_1 \beta e$ is a cycle, so we obtain (1) either way.
\end{proof}

\begin{cexample}
Just below, we have marked the weakly branching (but non-branching) vertices as red and marked every non-type (1) edge with its type in parentheses:
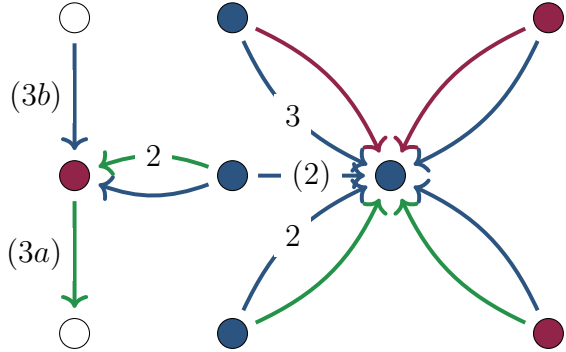
\begin{center}
\begin{tikzpicture}[scale=0.70]
 \SetUpEdge[lw         = 1.5pt,
            labelcolor = white]
  \tikzset{VertexStyle/.style = {draw,shape = circle,fill = white, inner sep=4pt,minimum size=10pt,outer sep=4pt}}

  \SetVertexNoLabel
    
  \Vertex[x=-6,y=3]{u6}

  \Vertex[x=-6,y=-3]{u10}
  
  \tikzset{VertexStyle/.style = {draw,shape = circle,fill = \niceblue, inner sep=4pt,minimum size=10pt,outer sep=4pt}}

  \Vertex[x=-3,y=0]{u2}
  \Vertex[x=0,y=0]{u3}

  \Vertex[x=-3,y=3]{u7}

  \Vertex[x=-3,y=-3]{u11}  
  
  \tikzset{VertexStyle/.style = {draw,shape = circle,fill = \nicered, inner sep=4pt,minimum size=10pt,outer sep=4pt}}

  \Vertex[x=-6,y=0]{u1}

  \Vertex[x=3,y=3]{u8}

  \Vertex[x=3,y=-3]{u12}

  \tikzset{EdgeStyle/.style = {->,color={\niceblue}}}
  \Edge[label=$(3b)$,labelstyle={left}](u6)(u1)
  \Edge[label=$(2)$](u2)(u3)

  \tikzset{EdgeStyle/.style = {->,color={\nicegreen}}}
  \Edge[label=$(3a)$,labelstyle={left}](u1)(u10)  
    
  \tikzset{EdgeStyle/.style = {->,bend left=20,color={\niceblue}}}
  \Edge[](u2)(u1)
  \Edge[](u8)(u3)
  \Edge[label=$2$](u11)(u3)
  
  \tikzset{EdgeStyle/.style = {->,bend right=20,color={\niceblue}}}
  \Edge[label=$3$](u7)(u3)
  \Edge[](u12)(u3)

  \tikzset{EdgeStyle/.style = {->,bend left=20,color={\nicegreen}}}
  \Edge[](u12)(u3)
  
  \tikzset{EdgeStyle/.style = {->,bend right=20,color={\nicegreen}}}
  \Edge[](u11)(u3)
  \Edge[label=$2$](u2)(u1)
  
  \tikzset{EdgeStyle/.style = {->,bend left=20,color={\nicered}}}
  \Edge[](u7)(u3)
  
  \tikzset{EdgeStyle/.style = {->,bend right=20,color={\nicered}}}
  \Edge[](u8)(u3)
\end{tikzpicture}
\end{center}
\end{cexample}\exend\medskip

Now we are ready to build a proper orientation of the branching subgraph.

\begin{theorem}\label{thm:SpGrFølner}
If $(E,C)$ is a finitely separated graph satisfying Condition \textup{(}N\textup{)}, then the branching subgraph $(E_{\textup{Br}},C^{\textup{Br}})$ admits a proper orientation. In particular, the tame $C^*$-algebra $\mathcal{O}(E_{\textup{Br}},C^{\textup{Br}})$ is nuclear.
\end{theorem}

\begin{proof}
To simplify the notation, let us assume that any vertex $u \in E^0$ satisfies $u \multimap v$ for some branching vertex $v$, i.e.~that $(E,C)$ is its own branching subgraph. Now let $v$ denote a branching vertex and let $e \in r^{-1}(v) \cup s^{-1}(v)$ be of type (1) as stated in Lemma~\ref{lem:Types}. We then first define an orientation $\mathfrak{o}_v(e)$ of $e$ relative to $v$ by
$$\mathfrak{o}_v(e):=\left\{ \begin{array}{cl}
-1 & \If \text{$v$ is of type (1) and either $e \in X_v$ or $s(e)=v$} \\
1 & \If \text{$v$ is of type (1), $r(e)=v$ and $e \not\in X_v$} \\
1 & \If \text{$v$ is of type (2) and either $e=e_v$ or $r(e)=v$} \\
-1 & \If \text{$v$ is of type (2), $s(e)=v$ and $e \ne e_v$}
\end{array} \right. $$
Note that if $e$ is a loop, i.e.~if $r(e)=s(e)$, then $r(e)$ is necessarily of type (1) and $e \in X_{r(e)}$, so the above is well-defined. Next, we define the orientation of arbitrary type (1) edges. Given any cycle $\alpha=e_n^{\varepsilon_n} \cdots e_1^{\varepsilon_1}$ based at $v$, observe that $\mathfrak{o}_v(e_1) = \varepsilon_1$ if and only if $\mathfrak{o}_v(e_n)=\varepsilon_n$ since $e_1^{\varepsilon_1}e_n^{\varepsilon_n}$ is admissible. We can therefore extend the orientation at $v$ by declaring
$$\mathfrak{o}_\alpha(e_i):=\left\{\begin{array}{cl}
\varepsilon_i & \If \mathfrak{o}_v(e_1)=\varepsilon_1 \\
- \varepsilon_i & \If \mathfrak{o}_v(e_1)=-\varepsilon_1
\end{array} \right. ,$$
and consequently $\mathfrak{o}_\alpha(e_i)=\mathfrak{o}_{\alpha^{-1}}(e_i)$. We claim that the orientation is in fact independent of the cycle in question, but in order to see this, we need the following claim: \medskip \\
\textbf{Claim:} If $e_2^{\varepsilon_2} \alpha e_1^{\varepsilon_1} \colon v \to u$ is an admissible path between branching vertices which may be extended to a cycle, then $\mathfrak{o}_v(e_1)=\varepsilon_1$ if and only if $\mathfrak{o}_u(e_2)=\varepsilon_2$.
\begin{proof}[Proof of claim]
Observe that in all situations, if $\mathfrak{o}_v(e_1) \ne \varepsilon_1$, then there is a closed path $\gamma$ based at $v$ making $e_2^{\varepsilon_2}\alpha e_1^{\varepsilon_1} \gamma e_1^{-\varepsilon_1} \alpha^{-1}e_2^{-\varepsilon_2}$ admissible. If $u$ is type (1), this implies $e_2^{\varepsilon_2} \in X_u$ so $\varepsilon_2=1 \ne -1=\mathfrak{o}_u(e_2)$, and if $u$ is type (2), then $e_2^{-\varepsilon_2}=e_u$ so $\varepsilon_2=-1 \ne 1=\mathfrak{o}_u(e_2)$. This proves one implication. To obtain the other, consider any extension $f_2^{\delta_2}\beta f_1^{\delta_1}e_2^{\varepsilon_2}\alpha e_1^{\varepsilon_1}$ to a cycle. If $\varepsilon_1=\mathfrak{o}_v(e_1)$, then $\delta_2=\mathfrak{o}_v(f_2)$ as well by the above, so $\mathfrak{o}_u(f_1)=\delta_1$ by the first implication. Using the above observation once more, we finally arrive at $\mathfrak{o}_u(e_2)=\varepsilon_2$ as well.
\end{proof}
Now if $\alpha=\alpha_2 e \alpha_1$ and $\beta=\beta_2e\beta_1$ are cycles based at branching vertices $v$ and $u$, respectively, then  $\alpha_2e\beta_1\beta_2e\alpha_1$ is a cycle based at $v$ passing through $u$, so evidently $\mathfrak{o}_\alpha(e)=\mathfrak{o}_\beta(e)$ by the above claim. We conclude that setting $\mathfrak{o}(e):=\mathfrak{o}_\alpha(e)$ for any choice of cycle $\alpha$ is well-defined. Finally, for edges of other types we simply set
$$\mathfrak{o}(e) := \left\{\begin{array}{cl}
-1 & \If \text{$e$ is of type (2) or (3a)} \\
1 & \If \text{$e$ is of type (3b)}
\end{array} \right. ,$$
so all that remains is to check the axioms.
 \medskip \\
First take $u \in E^0$ to be weakly branching; we then divide into the following two situations:
\begin{enumerate}
\item[(i)] There is some type (1) edge $e \in r^{-1}(u)$ with $\mathfrak{o}(e)=-1$.
\item[(ii)] There is some type (1) edge $e \in s^{-1}(u)$ with $\mathfrak{o}(e)=1$.
\end{enumerate} 
As any weakly branching vertex admits a cycle passing through a branching vertex, and every such cycle is either positively or negatively oriented by construction, clearly one of these will always hold. Observe also that both cannot hold at $u$: If there were type (1) edges $e \in r^{-1}(u)$ and $f \in s^{-1}(u)$ with $\mathfrak{o}(e)=-1$ and $\mathfrak{o}(f)=1$, then taking any positively oriented cycle $\alpha$ with $\ini_d(\alpha)=e^{-1}$, we would have $\ter_d(\alpha) \ne f^{-1}$. But then $u$ is branching, and the assumption contradicts the definition of $\mathfrak{o}$ at $u$. Hence (i) and (ii) are mutually exclusive. \medskip \\
Assuming (i), we claim that $u$ actually satisfies Definition~\ref{def:Or}(1). If there were some $f \in s^{-1}(v)$ with $\mathfrak{o}(f)=1$, then by the above $f$ would be of type (3b). But as $u=s(f)$ admits a cycle, this is surely not the case, hence $\mathfrak{o}^{-1}(1) \cap s^{-1}(u) = \emptyset$. Next, assume $[f]=[e]$; we must show that $\mathfrak{o}(f)=-1$. If $f$ is of type (1), there is a cycle $\alpha$ with $\ini_d(\alpha)=f^{-1}$. Taking any positively oriented cycle $\beta$ with $\ini_d(\beta)=e^{-1}$, we see that the cycle $\beta\alpha$ must be positively oriented, hence $\mathfrak{o}(f)=-1$. Now as the orientation of any type (2) or type (3a) edge is $-1$, it simply remains to check that $f$ is not of type (3b). Assuming that it is, there exists a closed path $\beta$ based at $u$ making $f^{-1}\beta f$ admissible. Taking some negatively oriented admissible path $\alpha \colon v \to u$ with $v$ branching and $\ter_d(\alpha)=e$, we then see that $\alpha^{-1}\beta\alpha$ is a closed path. However, this is impossible since $\alpha$ is negatively oriented. We conclude that $\mathfrak{o}(f)=-1$ as desired. Finally, we must show that $\mathfrak{o}(f)=1$ if $r(f)=u$ and $[f] \ne [e]$. Supposing first that $f$ is of type (1), we can take a cycle $\alpha$ with $\ini_d(\alpha)=f^{-1}$. Now if $\ter_d(\alpha) \in [e]$, then necessarily $\mathfrak{o}(f)=1$, and if $\ter_d(\alpha) \notin [e]$, then $u$ is branching, in which case the claim is clear. Next, we show that $f$ cannot be of type (2). If $u$ is branching, this is obvious, so let us assume it is not. Then we may take an admissible path $\alpha \colon u \to v$ with $v$ branching and $\ini_d(\alpha)=e^{-1}$ as well as a closed path $\beta$ based at $v$, such that $\alpha^{-1} \beta \alpha$ is admissible. Now observe that $f^{-1} \alpha^{-1} \beta \alpha f$ is a closed path, so $f$ is not of type (2). It is also clear that $f$ cannot be of type (3a), hence $\mathfrak{o}(f)=1$. We conclude that $\mathfrak{o}^{-1}(-1) \cap r^{-1}(u) = [e]$, so $\mathfrak{o}$ does indeed satisfy Definition~\ref{def:Or}(1) at $u$. Having done the harder case in details, we leave it to the reader to check that if (ii) holds for some weakly branching vertex $u$, then Definition~\ref{def:Or}(2) is satisfied at $u$. \medskip \\
Finally, assume that $u \in E^0$ is not weakly branching, and let $\alpha \colon u \to v$ and $\beta \colon v \to v$ be implementing the relation $u \multimap v$ for some branching vertex $v$. We once more divide into two different scenarios:
\begin{enumerate}
\item[(a)] The initial symbol satisfies $\ini_d(\alpha) \in (E^1)^{-1}$.
\item[(b)] The initial symbol satisfies $\ini_d(\alpha) \in E^1$.
\end{enumerate}
We only consider (a) and leave (b) to the reader. First observe that if $\alpha' \colon u \to v'$ and $\beta' \colon v' \to v'$ implement the relation $u \multimap v'$ for a branching vertex $v'$, then necessarily $\ini_d(\alpha') \in (E^1)^{-1}$ and $[\ini_d(\alpha)^{-1}] = [\ini_d(\alpha')^{-1}]$; otherwise $\alpha'^{-1}\beta'\alpha'\alpha^{-1}\beta\alpha$ would be a cycle passing through a branching vertex. Now whenever $[e]=[\ini_d(\alpha)^{-1}]$, we claim that such $\alpha'$ and $\beta'$ with $\ini_d(\alpha')=e^{-1}$ exist. Assuming this is not the case for one such edge $e$, there is an admissible path $\alpha' e \colon u \to v'$ and a cycle $\beta' \colon v' \to v'$ with $v'$ branching, such that $e^{-1}\alpha'^{-1}\beta \alpha'e$ is admissible. But then $\alpha'$ and $\beta'$ contradict the observation we have just made. We conclude that Definition~\ref{def:Or}(1) is indeed satisfied at $u$, thereby concluding the proof.
\end{proof}

\section{The Branch Free Subgraph}\label{sect:BF}

In this section, we study the subgraph obtained as the complement of the branching subgraph, and our ultimate goal is to prove that the associated $C^*$-algebra is nuclear.

\begin{definition}
Let $(E,C)$ denote a finitely separated graph. The \textit{branch free subgraph} $(E_{\textup{BF}},C^{\textup{BF}})$ is the full subgraph with vertex set $E_{\textup{BF}}^0:=E^0 \setminus E_{\textup{Br}}^0$, and the \textit{acyclic} subgraph $(E_{\textup{Ac}},C^{\textup{Ac}})$ is the full subgraph with vertex set
$$E_{\textup{Ac}}^0:=E^0 \setminus \{u \in E^0 \mid u \multimap v \text{ for some $v$}\}$$
Note that by definition, $E_{\textup{Ac}}^0 \subset E_{\textup{BF}}^0$.
\end{definition}

\begin{cexample}
This is the branch free subgraph of our example:
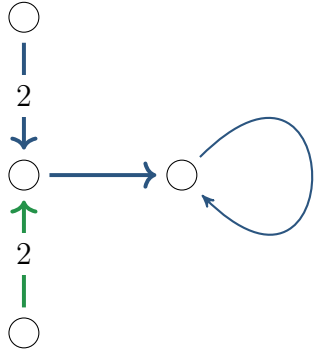
\begin{center}
\begin{tikzpicture}[scale=0.70]
 \SetUpEdge[lw         = 1.5pt,
            labelcolor = white]
  \tikzset{VertexStyle/.style = {draw,shape = circle,fill = white, inner sep=4pt,minimum size=10pt,outer sep=4pt}}

  \SetVertexNoLabel

  \Vertex[x=6,y=0]{u4}
  \Vertex[x=9,y=0]{u5}
  
  \Vertex[x=6,y=3]{u9}

  \Vertex[x=6,y=-3]{u13}  
  
  \tikzset{EdgeStyle/.style = {->,color={\niceblue}}}
  \Loop[dir=EA](u5)  
  \Edge[](u4)(u5)
  \Edge[label=$2$](u9)(u4)
    
  \tikzset{EdgeStyle/.style = {->,color={\nicegreen}}}
  \Edge[label=$2$](u13)(u4)
  
\end{tikzpicture}
\end{center}
Note that it cannot be given an orientation as one can 'change direction' using the doubled blue or green edges.
\exend
\end{cexample}

Before investigating these two subgraphs more closely, we introduce a very useful notion in the study of nuclearity.

\begin{definition}
Let $(E,C)$ denote a finitely separated graph. A hereditary and $C$-saturated subset $H \subset E^0$ will be called \textit{return free} if every admissible path $\alpha$ with $r(\alpha),s(\alpha) \in H$ is entirely contained in $(E_H,C^H)$.
\end{definition}

\begin{proposition}\label{prop:ReturnFree}
Let $(E,C)$ denote a finitely separated graph. If $H \in \mathcal{H}(E,C)$ is return free and both $\mathcal{O}(E_H,C^H)$ and $\mathcal{O}(E/H,C/H)$ are nuclear, then so is $\mathcal{O}(E,C)$.
\end{proposition}

\begin{proof}
By Theorem~\ref{thm:Quotient}, there is a short exact sequence
$$0 \rightarrow I(H) \rightarrow \mathcal{O}(E,C) \rightarrow \mathcal{O}(E/H,C/H) \rightarrow 0,$$
and since nuclearity passes to extensions, we simply have to verify that $I(H)$ is nuclear. Note that $\Omega(E,C)_H:= \bigsqcup_{v \in H}\Omega(E,C)_v$ is an $\mathbb{F}$-full clopen subspace of $U(H)$, so $I(H)$ is Morita equivalent to the crossed product $C(\Omega(E,C)_H) \rtimes \mathbb{F}$ for the restricted action $\theta^{(E,C)}\vert_{\Omega(E,C)_H}$ (see \cite[Definition 3.24]{AL} for the appropriate definitions and references). Since $H$ is assumed return free, the domain of any path $\alpha \in \mathbb{F} \setminus \mathbb{F}(E_H^1)$ is empty, hence
$$C(\Omega(E,C)_H) \rtimes \mathbb{F} \cong C(\Omega(E,C)_H) \rtimes \mathbb{F}(E_H^1).$$
Now observe that as in \cite[Remark 2.8]{Lolk}, there is a canonical d-bijective, $\mathbb{F}(E_H^1)$-equivariant surjective map $\Omega(E,C)_H \to \Omega(E_H,C^H)$ given by
$$\xi \mapsto \left\{\begin{array}{cl}
\xi \cap \mathbb{F}(E_H^1) &  \If \xi \in \Omega(E,C)_v \text{ for } v \notin (E_H)_{\textup{iso}}^0 \\
v & \If \xi \in \Omega(E,C)_v \text{ for } v \in (E_H)_{\textup{iso}}^0
\end{array}\right. .$$
Since the partial action $\theta^{(E_H,C^H)}$ is topologically amenable by assumption, so is the partial action $\mathbb{F}(E_H^1) \act \Omega(E,C)_H$ by Proposition~\ref{prop:DomainBij}. It follows that $C(\Omega(E,C)_H) \rtimes \mathbb{F}$ and, in turn, $I(H)$ is nuclear.
\end{proof}

With the above proposition in mind, the following lemma simplifies our task tremendously.

\begin{lemma}\label{lem:ReturnFree}
$E_{\textup{BF}}^0$ and $E_{\textup{Ac}}^0$ are return free hereditary and $C$-saturated sets for any finitely separated graph $(E,C)$.
\end{lemma}

\begin{proof}
We only consider the case of $E_{\textup{BF}}^0$ as the proofs are virtually identical, and we first verify that it is hereditary. Assuming that $s(e) \in E_{\textup{Br}}^0$ for some $e \in E^1$, there is an admissible path $\alpha \colon s(e) \to v$ and a cycle $\beta$ based at $v$ such that $\alpha^{-1}\beta\alpha$ is admissible. Setting $\gamma:= \alpha \cdot e^{-1}$, we see that $\gamma^{-1}\beta\gamma$ is admissible as well, hence $r(e) \in E_{\textup{Br}}^0$. \medskip \\
We move on to checking $C$-saturation, so let $u \in E_{\textup{Br}}^0$ and consider any $X \in C_u$. Take $\alpha$ and $\beta$ as above implementing the relation $u \multimap v$ for some branching vertex $v$. If $\ini_d(\beta\alpha) \in X^{-1}$, then we certainly have $s(X) \cap E_{\textup{Br}}^0 \ne \emptyset$, so we may assume that $\beta\alpha x$ is admissible for any $x \in X$. But then $(\alpha x)^{-1}\beta(\alpha x)$ implements the relation $s(x) \multimap v$, so in fact $s(X) \subset E_{\textup{Br}}^0$. We conclude that $E_{\textup{BF}}^0$ is indeed $C$-saturated. \medskip \\
Finally, we claim that $E_{\textup{BF}}^0$ is return free. Assume in order to reach a contradiction that there actually is an admissible path $\alpha$ with $r(\alpha),s(\alpha) \in E_{\textup{BF}}^0$, which is not contained in the branch free subgraph. Taking $\alpha$ to be minimal with these properties, we can write $\alpha=e_2^{-1} \beta e_1$ for $e_1,e_2 \in E^1  \setminus (E_{\textup{Br}}^1 \cup E_{\textup{BF}}^1)$ and $\beta$ an admissible path in the branching subgraph. Then we can take admissible paths $\alpha_i \colon r(e_i) \to v_i$, with $v_i$ a branching vertex, and cycles $\beta_i$ based at $v_i$ for $i=1,2$ making $\alpha_i^{-1}\beta_i\alpha_i$ admissible. As $s(e_i) \in E_{\textup{BF}}^0$, we must have $\ini_d(\alpha_i) \in [e_i]^{-1}$ for $i=1,2$. But then $(\alpha_2\beta e_1)^{-1}\beta_2(\alpha_2\beta e_1)$ is admissible, contradicting $v_1 \in E_{\textup{BF}}^0$.
\end{proof}

We easily obtain nuclearity in the acyclic case.

\begin{proposition}\label{prop:Acyclic}
The tame $C^*$-algebra $\mathcal{O}(E_{\textup{Ac}},C^{\textup{Ac}})$ is locally AF for any finitely separated graph $(E,C)$. In particular, it is nuclear.
\end{proposition}
\begin{proof}
We claim that $\mathcal{O}(E,C)$ is locally AF whenever $(E,C)$ admits no cycles, and by continuity of $\mathcal{O}$, we may assume $(E,C)$ to be finite. Now simply observe that any admissible path has length at most $3 \cdot \vert E^0 \vert$; otherwise it would contain a cycle. It follows that $\Omega(E,C)$ is a finite space and $\Omega(E,C)_\alpha \ne \emptyset$ for only finitely many $\alpha \in \mathbb{F}$, hence $\mathcal{O}(E,C)=C(\Omega(E,C)) \rtimes \mathbb{F}$ is finite dimensional.
\end{proof}

Now we are ready to deal with the final obstacle before the main theorem.

\begin{theorem}\label{thm:CoNucl}
$\mathcal{O}(E_{\textup{BF}},C^{\textup{BF}})$ is nuclear for any finitely separated graph $(E,C)$.
\end{theorem}
\begin{proof}
We claim that $\mathcal{O}(E,C)$ is nuclear whenever $(E,C)$ is a finitely separated graph without branching vertices, and by continuity of $\mathcal{O}$, we might as well assume $(E,C)$ to be finite. Moreover, as $\mathcal{O}(E_{\textup{Ac}},E^{\textup{Ac}})$ is nuclear by Proposition~\ref{prop:Acyclic}, we may further reduce to the quotient graph $(E/H,C/H)$ for $H:=E_{\textup{Ac}}^0$ due to Proposition~\ref{prop:ReturnFree} and Lemma~\ref{lem:ReturnFree}. In conclusion, we may take $(E,C)$ to be a finite graph such that for all $u \in E^0$, $u \multimap v$ for some $v \in E^0$. \medskip \\
Let $V \subset E^0$ denote the set of vertices admitting a cycle and write $v_1 \sim v_2$ for $v_1,v_2 \in V$ if there is a cycle passing through both $v_1$ and $v_2$. This is clearly an equivalence relation, and we will prove the theorem by induction over $\vert \bigslant{V}{\sim} \vert$, the number of equivalence classes. For the induction start, assume $\vert \bigslant{V}{\sim} \vert=1$. We claim that for any $\xi \in \Omega(E,C)$, viewing it as a tree, there is a bi-infinite linear subset $F^\xi \subset \xi$ with the following properties
\begin{enumerate}
\item $F^{\theta_\alpha(\xi)}=F^\xi \cdot \alpha^{-1}$ whenever $\alpha \in \xi$,
\item $\text{dist}(F^\xi,1) \le 3 \vert E^0 \vert$ in the ordinary tree metric of $\xi$,
\item $\xi \mapsto F_n^\xi := F^\xi \cap \xi^{n+3\vert E^0 \vert}$ is locally constant for any $n$,
\end{enumerate}
where $\xi^k:=\{\beta \in \xi \colon \vert \beta \vert \le k\}$. It should be clear from these properties that $(F_n)$ will be a topological F{\o}lner sequence, hence $\mathcal{O}(E,C)$ will be nuclear by Theorem~\ref{thm:TopAmenThm} and Proposition~\ref{prop:FolnerImpliesTopAmen}. Specifically, define
$$F^\xi := \{\alpha \in \xi \mid r(\alpha) \in V\},$$
where $r(1):=v$ for $v$ such that $\xi \in \Omega(E,C)_v$ by convention. In order to see that $F^\xi$ enjoys the above properties, we first need to verify the following minor claim. \medskip \\
\textbf{Claim:} Let $X \in C$. If for some $e \in X$ there exists an admissible path $\gamma$ with $\ini_d(\gamma)=e^{-1}$ and $r(\gamma) \in V$, then any $f \in X$ has this property.
\begin{proof}[Proof of claim]
Let $e \in X$ and $\gamma$ be as above, and assume in order to reach a contradiction that some $f \in X$ does not have the above property. Then, since $s(f) \multimap v$ for some $v \in V$, there is an admissible path $\alpha f \colon s(f) \to v$ and a cycle $\beta \colon v \to v$ such that $\alpha^{-1}\beta\alpha$ is admissible. Now by $v \sim r(\gamma)$, the admissible path $\gamma \alpha^{-1}$ can be extended to a closed path based at $v$. But from $\alpha^{-1}\beta\alpha$ being admissible, we may then conclude that $v$ is branching, a contradiction.
\end{proof}

One immediate consequence of the claim (and the assumption that any $u$ satisfies $u \multimap v$ for some $v$) is that any $\xi \in \Omega(E,C)$ is necessarily infinite. Now observe that if an admissible path passes a vertex three times (including the source and range), then that vertex must admit a cycle. In particular, there is some $\alpha \in F^\xi$ with $\vert \alpha \vert \le 3 \vert E^0 \vert$. Another consequence of the claim is that for all $\alpha \in F^\xi$, at least two neighbours of $\alpha$ in $\xi$ are contained in $F^\xi$. It therefore only remains to check that $F^\xi$ is contained a linear subset. Assume in order to reach a contradiction that this is not the case for some $\xi \in \Omega(E,C)$. Then, by possibly replacing $\xi$ with a translate, we can find non-trivial admissible paths $\alpha_1,\alpha_2,\alpha_3 \in \xi$ with $r(\alpha_i) \in V$ such that the concatenation $\alpha_i^{-1}\alpha_j$ is admissible for $i \ne j$. We claim that $u:=s(\alpha_i) \in V$. If this were not the case, then we could take a non-trivial admissible path $\alpha \colon u \to v$ and a cycle $\beta$ at $v$ implementing the relation $u \multimap v$ for some $v \in V$. We may then apply the same argument as above in the proof of the claim to show that $v$ is branching, a contradiction. Now since $u \in V$ and $u \sim r(\alpha_i)$ for $i=1,2,3$, each $\alpha_i$ may be extended to a closed path. But then $u$ is branching -- we conclude that $F^\xi$ is indeed a bi-infinite linear subset. It is obvious that it satisfies (1) and (3), and we have already seen that (2) holds as well, thereby concluding the proof of the induction start. \medskip \\
For the inductive step, let $\vert \bigslant{V}{\sim} \vert \ge 2$ and suppose that the claim holds whenever the number of equivalence classes is at most $\vert \bigslant{V}{\sim} \vert-1$. We let the relation $\multimap$ descend to the equivalence classes $\mathfrak{u},\mathfrak{v} \in \bigslant{V}{\sim}$ by
$$ \mathfrak{u} \multimap \mathfrak{v} \Leftrightarrow \text{$u \multimap v$ for some $u \in \mathfrak{u}$ and $v \in \mathfrak{v}$},$$
and note that if $u \multimap v$, then $u' \multimap v$ for any $u' \sim u$. Now observe that $\multimap$ becomes antisymmetric on the set of equivalence classes, so there exist $\mathfrak{u},\mathfrak{v} \in \bigslant{V}{\sim}$ with $\mathfrak{u} \not\multimap \mathfrak{v}$. Defining $H:=\{u \in E^0 \mid u \not\multimap v \text{ for any } v \in \mathfrak{v}\}$, it is then straightforward to check that $H$ is a return free hereditary and $C$-saturated set. Being return free, any cycle in $(E,C)$ is properly contained in one of the graphs $(E_H,C^H)$ and $(E/H,C/H)$, and since $\mathfrak{u} \subset H$ and $\mathfrak{v} \subset (E/H)^0$, the two graphs satisfy the inductive assumption. We conclude that $\mathcal{O}(E_H,C^H)$ and $\mathcal{O}(E/H,C/H)$ are both nuclear, hence so is $\mathcal{O}(E,C)$ by Proposition~\ref{prop:ReturnFree}. This finishes the inductive step and, in turn, the proof.
\end{proof}

\begin{cexample}
Decomposing our branch free subgraph as in the above proof will leave us with the two subgraphs
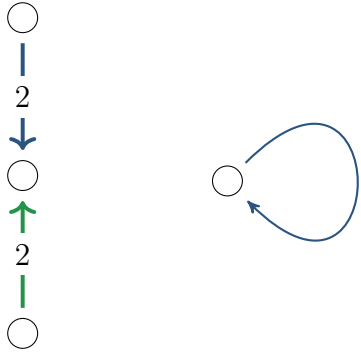
\begin{center}
\begin{tikzpicture}[scale=0.70]
 \SetUpEdge[lw         = 1.5pt,
            labelcolor = white]
  \tikzset{VertexStyle/.style = {draw,shape = circle,fill = white, inner sep=4pt,minimum size=10pt,outer sep=4pt}}

  \SetVertexNoLabel

  \Vertex[x=6,y=0]{u4}
  
  \Vertex[x=6,y=3]{u9}

  \Vertex[x=6,y=-3]{u13}  
  
  \tikzset{EdgeStyle/.style = {->,color={\niceblue}}}
  \Edge[label=$2$](u9)(u4)
    
  \tikzset{EdgeStyle/.style = {->,color={\nicegreen}}}
  \Edge[label=$2$](u13)(u4)
  
\end{tikzpicture} \hspace{2cm} \begin{tikzpicture}[scale=0.70]
 \SetUpEdge[lw         = 1.5pt,
            labelcolor = white]
  \tikzset{VertexStyle/.style = {draw,shape = circle,fill = white, inner sep=4pt,minimum size=10pt,outer sep=4pt}}

  \SetVertexNoLabel

  \Vertex[x=9,y=0]{u5}
    
  \tikzset{EdgeStyle/.style = {->,color={\niceblue}}}
  \Loop[dir=EA](u5)  
  
\end{tikzpicture}
\end{center}
each of which produces nuclear $C^*$-algebras.
\end{cexample}

\begin{remark}
In the above proof, we relied on an inductive argument to prove that for a finite graph without branching vertices $(E,C)$, the quotient graph $(E/H,C/H)$ for $H:=E_{\textup{Ac}}^0$ gives rise to a nuclear $C^*$-algebra. It is also possible to prove this claim directly by construction a very natural topological F{\o}lner sequence $(F_n)$, where $F_n^\xi$ is simply the $n$-ball $\xi^n:=\{ \alpha \in \xi \colon \vert \alpha \vert \le 1 \}$. However, we prefer the above proof due to the noteworthy technical difficulties that arise when verifying the F\o lner property of this sequence.
\end{remark}

\section{The main theorem and examples}

In this final section, we give a short proof of the main theorem and finally consider a few examples of Condition (N) graphs and their graph monoids. \medskip \\
Note that quite a few of the implications in the main theorem follow from general theory. Indeed, the implications $(3) \Leftrightarrow (4) \Rightarrow (2)$ are immediate from Theorem~\ref{thm:TopAmenThm}, which is an application of the groupoid $C^*$-algebra theory of Renault and Anantharaman-Delaroche, while $(2) \Rightarrow (5)$ follows from Exel's theory of partial crossed products, specifically \cite[Proposition 25.12]{Exel}, since free groups are exact \cite[Proposition 5.1.8]{BO}.
\begin{theorem}\label{thm:MainThm}
For any finitely separated graph $(E,C)$, the following are equivalent:
\begin{enumerate}
\item[\textup{(1)}] $(E,C)$ satisfies Condition \textup{(}N\textup{)}.
\item[\textup{(2)}] The regular representation $\mathcal{O}(E,C) \to \mathcal{O}^r(E,C)$ is an isomorphism.
\item[\textup{(3)}] The $C^*$-algebra $\mathcal{O}^r(E,C)$ is nuclear.
\item[\textup{(4)}] The $C^*$-algebra $\mathcal{O}(E,C)$ is nuclear.
\item[\textup{(5)}] The $C^*$-algebra $\mathcal{O}(E,C)$ is exact.
\item[\textup{(6)}] Every stabiliser of the partial action $\theta^{(E,C)}$ is amenable \textup{(}hence trivial or cyclic\textup{)}.
\end{enumerate}
\end{theorem}

\begin{proof}
We have already proven $(5) \Rightarrow (6) \Rightarrow (1)$ in Proposition~\ref{prop:Exactness}, so in light of the above comments, it remains only to verify $(1) \Rightarrow (4)$. But by Proposition~\ref{prop:ReturnFree} and Lemma~\ref{lem:ReturnFree}, it suffices to show that $\mathcal{O}(E_{\textup{Br}},C^{\textup{Br}})$ and $\mathcal{O}(E_{\textup{BF}},C^{\textup{BF}})$ are nuclear, and this was done in Theorem~\ref{thm:SpGrFølner} and Theorem~\ref{thm:CoNucl}, respectively.
\end{proof}

We obtain the following purely graph-theoretic consequence of the above theorem. Recall from \cite{AE} that if $(E,C)$ is finite and bipartite, then one can associate a sequence $(E_n,C^n)$ of finite bipartite graphs to $(E,C)=(E_0,C^0)$, such that $\mathcal{O}(E,C) \cong \varinjlim_n C^*(E_n,C^n)$ for certain connecting homomorphisms.

\begin{corollary}\label{cor:CondNBip}
Let $(E,C)$ denote a finite bipartite graph. Then $(E_n,C^n)$ satisfies Condition \textup{(}N\textup{)} for some $n$ if and only if it does so for all $n$.
\end{corollary}
\begin{proof}
This is immediate from Theorem~\ref{thm:MainThm} since $\mathcal{O}(E_m,C^m) \cong \mathcal{O}(E_n,C^n)$ for all $m,n$.
\end{proof}

We finally consider some examples.

\begin{example}
If $\vert  s^{-1}(v) \vert + \vert C_v \vert \le 2$ for any $v \in E^0$, then $(E,C)$ contains no branching vertices and therefore satisfies Condition (N) trivially. This covers examples such as
\begin{center}
\begin{tikzpicture}[scale=0.6]
 \SetUpEdge[lw         = 1.5pt,
            labelcolor = white]
  \tikzset{VertexStyle/.style = {draw,shape = circle,fill = white, inner sep=4pt,minimum size=10pt,outer sep=4pt}}

  \SetVertexNoLabel
  \Vertex[x=0,y=0]{u}
  \Vertex[x=-3,y=-3]{v1}
  \Vertex[x=3,y=-3]{v2}  
  
  \tikzset{EdgeStyle/.style = {->,bend left,color={\nicered}}}
  \Edge[](v1)(u)
  \tikzset{EdgeStyle/.style = {->,bend right,color={\nicered}}}
  \Edge[](v1)(u)    
  \tikzset{EdgeStyle/.style = {->,bend left,color={\niceblue}}}
  \Edge[](v2)(u)  
  \tikzset{EdgeStyle/.style = {->,bend right,color={\niceblue}}}  
  \Edge[](v2)(u)  
\end{tikzpicture} \hspace{0.5cm} \begin{tikzpicture}[scale=0.6]
 \SetUpEdge[lw         = 1.5pt,
            labelcolor = white]
  \tikzset{VertexStyle/.style = {draw,shape = circle,fill = white, inner sep=4pt,minimum size=10pt,outer sep=4pt}}

  \SetVertexNoLabel
  \Vertex[x=0,y=0]{u}
  \Vertex[x=-3,y=-3]{v1}
  \Vertex[x=0,y=-3]{v3}
  \Vertex[x=3,y=-3]{v2}

  \tikzset{EdgeStyle/.style = {->,bend left,color={\nicered}}}
  \Edge[](v1)(u)
  \Edge[](v3)(u)    
  \tikzset{EdgeStyle/.style = {->,bend right,color={\niceblue}}}  
  \Edge[](v2)(u)  
  \Edge[](v3)(u)  
\end{tikzpicture}
\hspace{0.5cm}
\begin{tikzpicture}[scale=0.6]
 \SetUpEdge[lw         = 1.5pt,
            labelcolor = white]
  \tikzset{VertexStyle/.style = {draw,shape = circle,fill = white, inner sep=4pt,minimum size=10pt,outer sep=4pt}}

  \SetVertexNoLabel
  \Vertex[x=0,y=0]{u}
  \Vertex[x=-3,y=-3]{v1}
  \Vertex[x=3,y=-3]{v2}  
  
  \tikzset{EdgeStyle/.style = {->,bend left,color={\nicered}}}
  \Edge[](v1)(u)
  \Edge[](v2)(u)    
  \tikzset{EdgeStyle/.style = {->,bend right,color={\niceblue}}}  
  \Edge[](v1)(u)  
  \Edge[](v2)(u)  
\end{tikzpicture},
\end{center}
which were considered in \cite[Example 9.5, 9.6, 9.7]{AE} as well as \cite[Example 6.7]{AL}.
\end{example}

\begin{example}
If $E$ is any column-finite graph, then $E$ clearly satisfies Condition (N) when regarded as a trivially separated graph. In fact, $\mathfrak{o}(e):=-1$ for all $e \in E^1$ defines an orientation, but unless $E$ contains no sources, it is not a proper orientation. This can be circumvented in two different ways: Either one mods out by the acyclic subgraph as we have done above, or one adds heads to all sources. Either way, the new graph will admit a proper orientation, giving another proof of nuclearity of classical graph $C^*$-algebras of column-finite graphs.
\end{example}

\begin{example}
If $(E,C):=(E(m,n),C(m,n))$ for some $2 \le m \le n <\infty$ is the graph

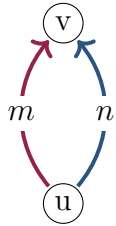
\begin{center}
\begin{tikzpicture}[scale=0.8]
 \SetUpEdge[lw         = 1.5pt,
            labelcolor = white]
  \tikzset{VertexStyle/.style = {draw, shape = circle,fill = white, minimum size=15pt, inner sep=2pt,outer sep=1pt}}

  \Vertex[x=0,y=0]{u}
  \Vertex[x=0,y=3]{v}

  \tikzset{EdgeStyle/.style = {->,bend left=40,color={\nicered}}}  
  \Edge[label=$m$](u)(v)

  \tikzset{EdgeStyle/.style = {->,bend right=40,color={\niceblue}}}  
  \Edge[label=$n$](u)(v)

\end{tikzpicture}
\end{center}
as in \cite[Example 9.3]{AE}, so that $(E,C)$ gives rise to the $(m,n)$-dynamical systems of \cite{AEK}, then $u$ is a branching vertex without a local orientation. We therefore recover the fact from \cite[Theorem 7.2]{AEK} that $\mathcal{O}_{m,n}:=\mathcal{O}(E,C)$ is non-exact.  \exend
\end{example}

For every finitely separated graph $(E,C)$, there is a natural \textit{graph monoid} $M(E,C)$ as defined in \cite{AG2}: It is the universal abelian monoid with generators $E^0$ and relations $v = \sum_{e \in X}s(e)$ for all $v \in E^0$ and $X \in C_v$. By \cite[Proposition 4.4]{AG2}, every conical abelian monoid can be represented as $M(E,C)$ for an appropriate finitely separated bipartite graph $(E,C)$. Conversely, the graph monoid $M(E)$ of a non-separated column-finite graph is always quite well-behaved: Whenever $E$ is finite, $M(E)$ is a refinement monoid by \cite[Proposition 4.4]{AMP}, hence primely generated by \cite[Corollary 6.8]{Brookfield}. It follows that $M(E)$ satisfies the extensive list of properties given by \cite[Theorem 5.19]{Brookfield}, of which many pass to direct limits. Since any column-finite graph $E$ is a direct limit of its finite complete subgraphs and the assignment $E \mapsto M(E)$ is continuous by \cite[Lemma 3.4]{AMP}, the monoid $M(E)$ satisfies all such properties. Among these are
\begin{itemize}
\item \textit{Unperforation}: If $n \cdot a \le n \cdot b$, then $a \le b$ for all integers $n \ge 2$.
\item \textit{Pseudo-cancellation}: If $a + c \le b + c$, then there is some $a_1$ with $a_1 + c \le c$ for which $a \le b+ a_1$.
\end{itemize}
Below we shall see two basic examples of finite bipartite separated graphs satisfying Condition (N) for which the graph monoids do not enjoy unperforation and pseudo-cancellation, respectively.
\medskip \\
Recall that $M(E) \cong \mathcal{V}(L_K(E)) \cong \mathcal{V}(C^*(E))$ by \cite[Theorem 3.5 and Theorem 7.1]{AMP}. In the separated setting, we have $M(E,C) \cong \mathcal{V}(L_K(E,C))$ due to \cite[Theorem 4.3]{AG2}, and the quotient map $L_K(E,C) \to L_K^{\textup{ab}}(E,C)$ induces a refinement $\mathcal{V}(L_K(E,C)) \to \mathcal{V}(L_K^{\textup{ab}}(E,C))$ by \cite[Corollary 5.9]{AE}. However, it is still an open problem whether the inclusion $L_\mathbb{C}(E,C) \hookrightarrow C^*(E,C)$ induces an isomorphism $\mathcal{V}(L_\mathbb{C}(E,C)) \to \mathcal{V}(C^*(E,C))$. If this happens to be the case, then $\mathcal{V}(L_\mathbb{C}^{\textup{ab}}(E,C)) \cong \mathcal{V}(\mathcal{O}(E,C))$ as well by \cite[Theorem 5.7]{AE}, in which case there will be a natural refinement $M(E,C) \to \mathcal{V}(\mathcal{O}(E,C))$.

\begin{example}\label{ex:Unperforation}
Consider the Condition (N) graph
\begin{center}
\begin{tikzpicture}[scale=0.6]
 \SetUpEdge[lw         = 1.5pt,
            labelcolor = white]
  \tikzset{VertexStyle/.style = {draw,shape = circle,fill = white, inner sep=4pt,minimum size=10pt,outer sep=4pt}}

  \Vertex[x=0,y=0]{w}
  \Vertex[x=-3,y=-3]{u}
  \Vertex[x=3,y=-3]{v}  
  
  \tikzset{EdgeStyle/.style = {->,bend left,color={\nicered}}}
  \Edge[label=$e_0$, style={circle,inner sep=0pt }](u)(w)
  \tikzset{EdgeStyle/.style = {->,bend right,color={\nicered}}}
  \Edge[label=$e_1$, style={circle,inner sep=0pt }](u)(w)    
  \tikzset{EdgeStyle/.style = {->,bend left,color={\niceblue}}}
  \Edge[label=$f_0$, style={circle,inner sep=0pt }](v)(w)  
  \tikzset{EdgeStyle/.style = {->,bend right,color={\niceblue}}}  
  \Edge[label=$f_1$, style={circle,inner sep=0pt }](v)(w)  
\end{tikzpicture}
\end{center}
of \cite[Example 9.5]{AE} and denote it by $(E,C)$. Then
$$M(E,C) = \langle u,v,w \mid 2u=w=2v \rangle \cong \langle u, v \mid 2u=2v \rangle,$$
but $u \not\le v$ and $v \not\le u$, so $M(E,C)$ is not unperforated. It follows from \cite[Theorem 7.4]{AE} that the relative type semigroup $S(\Omega(E,C),\mathbb{F},\mathbb{K})$ of the partial action $\theta^{(E,C)}$ is isomorphic to $\mathcal{V}(L_K^{\textup{ab}}(E,C))$, so in particular, $M(E,C)$ unitarily embeds into the type semigroup, and this will be unperforated as well. We will now identify the partial action $\theta^{(E,C)}$ up to Kakutani equivalence (see \cite[Definition 2.14]{Li2}) with a concrete partial $\mathbb{F}_2$-action. Consider the sequence space $\mathcal{X}:=\{0,1\}^\mathbb{Z}$ along with the clopen subspaces
$$\mathcal{X}_{i \LargerDot}:=\{x \in \mathcal{X} \mid x_{-1}=i\} \andspace \mathcal{X}_{\LargerDot i}:=\{x \in \mathcal{X} \mid x_0=i\}$$
for $i=0,1$. We then define partial homeomorphisms $\varphi_L \colon X_{0\LargerDot} \to X_{1\LargerDot}$ and $\varphi_R \colon X_{\LargerDot 0} \to X_{\LargerDot1}$ given by
\begin{align*}
\varphi_L(\ldots x_{-3}x_{-2}0\LargerDot x_0x_1x_2 \ldots)&:=(\ldots x_{1}x_{0}1 \LargerDot x_{-2}x_{-3}x_{-4} \ldots) \\
\varphi_R(\ldots x_{-3}x_{-2}x_{-1}\LargerDot 0x_1x_2 \ldots)&:=(\ldots x_{3}x_{2}x_{1} \LargerDot 1x_{-1}x_{-2}x_{-3} \ldots)
\end{align*}
and consider the semi-saturated partial action $\theta \colon \mathbb{F}[L,R] \act \mathcal{X}$ induced by $\varphi_L$ and $\varphi_R$. Then $\theta$ is quasi-conjugate (see \cite[Definition 2.2]{Lolk}) to the restriction $\theta^{(E,C)}\vert_\Omega$ of $\theta^{(E,C)}$ to the full (see \cite[Definition 3.24]{AL}) clopen subspace $\Omega:=\Omega(E,C)_w$ as follows. Observe first that every configuration $\xi \in \Omega$ may be represented by an ordered pair $(\xi^{red},\xi^{blue})$ of infinite admissible paths
$$ \xi^{red} := \ldots f_{\underline{i_{-4}}} f_{i_{-4}}^{-1} e_{\underline{i_{-3}}}e_{i_{-3}}^{-1}f_{\underline{i_{-2}}} f_{i_{-2}}^{-1}e_{\underline{i_{-1}}}e_{i_{-1}}^{-1} \andspace \xi^{blue}:=\ldots e_{\underline{i_3}} e_{i_3}^{-1} f_{\underline{i_2}}f_{i_2}^{-1} e_{\underline{i_1}}e_{i_1}^{-1}f_{\underline{i_0}}f_{i_0}^{-1} \ ,$$
for $i_k \in \{0,1\}$, where $\underline{0}:=1$ and $\underline{1}:=0$. $\xi^{red}$ represents the part of the configuration that initially travels by the inverse of a red edge, while $\xi^{blue}$ initially travels by the inverse of a blue edge. An element $\alpha \in \mathbb{F}$ may act on $\xi$ if and only if $\alpha \le \xi^{red}$ or $\alpha \le \xi^{blue}$, and the result $\theta^{(E,C)}_\alpha(\xi)$ lies in $\Omega$ if and only if $\alpha$ is of even length. If $\alpha$ denotes the simple cycle $\alpha=e_{\underline{i_{-1}}}e_{-1}^{-1}$, we have
$$\theta^{(E,C)}_\alpha(\xi)^{red}=\ldots f_{\underline{i_2}}f_{i_2}^{-1} e_{\underline{i_1}}e_{i_1}^{-1}f_{\underline{i_0}}f_{i_0}^{-1}e_{i_{-1}}e_{\underline{i_{-1}}}^{-1} \andspace \theta^{(E,C)}_\alpha(\xi)^{blue}=\ldots f_{\underline{i_{-4}}} f_{i_{-4}}^{-1} e_{\underline{i_{-3}}}e_{i_{-3}}^{-1}f_{\underline{i_{-2}}} f_{i_{-2}}^{-1},$$
and a similar situation occurs for $\alpha=f_{\underline{i_0}}f_{i_0}^{-1}$. Now consider the group homomorphism $\Psi \colon \mathbb{F}[L,R] \to \mathbb{F}$ given by $\Psi(L)= e_1^{-1}e_0$ and $\Psi(R)= f_1^{-1}f_0$, as well as the homeomorphism $\psi \colon \mathcal{X} \to \Omega$ defined by $\psi(x)= (\xi_x^{red},\xi_x^{blue})$, where
$$ \xi_x^{red} := \ldots f_{\underline{x_{-4}}} f_{x_{-4}}^{-1} e_{\underline{x_{-3}}}e_{x_{-3}}^{-1}f_{\underline{x_{-2}}} f_{x_{-2}}^{-1}e_{\underline{x_{-1}}}e_{x_{-1}}^{-1} \andspace \xi_x^{blue}:=\ldots e_{\underline{x_3}} e_{x_3}^{-1} f_{\underline{x_2}}f_{x_2}^{-1} e_{\underline{x_1}}e_{x_1}^{-1}f_{\underline{x_0}}f_{x_0}^{-1}.$$
It is clear from our above observation that the pair $(\psi, \Psi)$ defines a quasi-conjugacy $\theta \to \theta^{(E,C)}\vert_{\Omega}$, so $\theta$ and $\theta^{(E,C)}$ are indeed Kakutani equivalent.
\end{example}

\begin{example}\label{ex:PseudoCancellation}
Finally, consider the graph
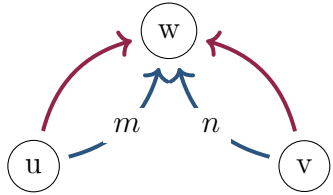
\begin{center}

\begin{tikzpicture}[scale=0.6]
 \SetUpEdge[lw         = 1.5pt,
            labelcolor = white]
  \tikzset{VertexStyle/.style = {draw,shape = circle,fill = white, inner sep=4pt,minimum size=10pt,outer sep=4pt}}

  \Vertex[x=0,y=0]{w}
  \Vertex[x=-3,y=-3]{u}
  \Vertex[x=3,y=-3]{v}  
  
  \tikzset{EdgeStyle/.style = {->,bend left,color={\nicered}}}
  \Edge[](u)(w)
  \tikzset{EdgeStyle/.style = {->,bend right,color={\niceblue}}}
  \Edge[label=$m$](u)(w)    
  \tikzset{EdgeStyle/.style = {->,bend left,color={\niceblue}}}
  \Edge[label=$n$](v)(w) 
  \tikzset{EdgeStyle/.style = {->,bend right,color={\nicered}}}  
  \Edge[](v)(w)  
\end{tikzpicture}
\end{center}
with $m \ge 2$, $n \ge 1$, and denote it by $(E,C)$. Then
$$M(E,C) = \langle u,v,w \mid w=u+v=m \cdot u+ n \cdot v \rangle \cong \langle u,v \mid u+v = m \cdot u+n \cdot v \rangle,$$
and we claim that $M(E,C)$ does not enjoy pseudo-cancellation as above with 
$$a:=(m-1) \cdot u, \quad b:=v \andspace c:=u.$$
These elements surely satisfy the assumption, namely that 
$$a+c=m \cdot u \le m \cdot u + n \cdot v = u+v = b+c,$$
so if pseudo-cancellation were present, then there should exist some $a_1 \in M(E,C)$ with $a_1+c \le c$ and $a \le b+a_1$. However, the first inequality entails $a_1=0$, hence
$$(m-1) \cdot u = a \le b = v,$$
which is absurd. It follows that the type semigroup $S(\Omega(E,C),\mathbb{F},\mathbb{K})$ is not pseudo-cancellative either. Contrary to Example~\ref{ex:Unperforation}, we have no simple description of the partial action $\theta^{(E,C)}$ as the configurations are much more complicated in this case. \exend
\end{example}

Example~\ref{ex:Unperforation} and Example~\ref{ex:PseudoCancellation} show that one should expect to find nuclear tame separated graph $C^*$-algebras of a rather different nature than that of classical graph $C^*$-algebras, in particular with a more general projection structure. However, while both unperforation and pseudo-cancellation may fail in the graph monoid $M(E,C)$ of a Condition (N) graph, the author expects that it will always enjoy the following important cancellation properties:
\begin{itemize}
\item\textit{Almost unperforation}: If $(n+1) \cdot a \le n \cdot b$ for some $n \ge 2$, then $a \le b$.
\item\textit{Separation}: If $2a=a+b=2b$, then $a=b$.
\end{itemize}
Observe that if any monoid theoretic property, which passes to limits, holds for all finite bipartite Condition (N) graphs $(E,C)$, then it will automatically hold for 
$$S(\Omega(E,C),\mathbb{F},\mathbb{K}) \cong \mathcal{V}(L^{ab}(E,C)) \cong \varinjlim_n M(E_n,C^n)$$
as well by Corollary~\ref{cor:CondNBip} and \cite[Corollary 5.9]{AE}. As a consequence, the author expects $\theta^{(E,C)}$ to not be topologically amenable whenever it presents a counterexample to the topological version of Tarski's theorem considered in \cite[Section 7]{AE}.

\section*{Acknowledgements}
The author would like to thank Pere Ara for many fruitful discussions related to this work, part of which was carried out during a stay at Universitat Aut\`onoma de Barcelona.

\bibliographystyle{plain}
\bibliography{ref}

\end{document}